\documentclass{gtpart}     
\usepackage[margin=1in]{geometry}
\title{The Asymptotic Behavior of Finite Energy Symplectic Vortices with Admissible Metrics}

\author{Bohui Chen}
\givenname{Bohui}
\surname{Chen}
\address{Department of Mathematics and Yangtz center of Mathematics, Sichuan University, Chengdu, 610065, China}
\email{bohui@cs.wisc.edu}
\urladdr{}
\author{Bai-Ling Wang}
\givenname{Bai-Ling}
\surname{Wang}
\address{Department of Mathematics, Mathematical Sciences Institute, Australian National University, Canberra, ACT 0200, Australia}
\email{bai-ling.wang@anu.edu.au}
\urladdr{}
\author{Rui Wang}
\givenname{Rui}
\surname{Wang}
\address{Department of Mathematics, University of California, Irvine, CA, 92697-3875, USA}
\email{ruiw10@math.uci.edu}
\urladdr{}

\keyword{Symplectic reduction, Hamiltonian Gromov--Witten, Symplectic vortex equation, Asymptotic behavior, Admissible metric}

\arxivreference{}
\arxivpassword{}

\usepackage{graphicx, color}
\newcommand{\blue}{\textcolor{blue}}

\newtheorem{thm}{Theorem}[section]    
\newtheorem{lem}[thm]{Lemma}          
\theoremstyle{definition}
\newtheorem{defn}[thm]{Definition}    
\newtheorem{rem}[thm]{Remark} 
\newcommand{\ben}{\begin{enumerate}}
\newcommand{\een}{\end{enumerate}}
\newcommand{\be}{\begin{equation}}
\newcommand{\ee}{\end{equation}}
\newcommand{\bea}{\begin{eqnarray}}
\newcommand{\eea}{\end{eqnarray}}
\newcommand{\beastar}{\begin{eqnarray*}}
\newcommand{\eeastar}{\end{eqnarray*}}
\newcommand{\bc}{\begin{center}}
\newcommand{\ec}{\end{center}}

\theoremstyle{theorem}
\newtheorem{cor}[thm]{Corollary}
\newtheorem{prop}[thm]{Proposition}

\theoremstyle{definition}

\def\Z{{\mathbb Z}}
\def\C{{\mathbb C}}
\def\R{{\mathbb R}}

\def\C{\mathbb{C}}
\def\Z{\mathbb{Z}}

\def\CA{{\mathcal A}}
\def\CB{{\mathcal B}}

\def\CE{{\mathcal E}}
\def\CF{{\mathcal F}}
\def\CG{{\frak g}}

\def\CI{{\mathcal I}}

\def\CL{{\mathcal L}}

\def\CU{{\mathcal U}}
\def\CV{{\mathcal V}}

\def\CX{{\mathcal X}}

\def\delbar{{\overline \partial}}
\def\del{{\partial}}
\def\Lie{\textup{Lie}}

\def\Aut{\textup{Aut}}

\def\exp{\textup{exp}}
\def\id{\textup{id}}
\def\pr{\textup{pr}}

\def\Ad{\textup{Ad}}

\def\Crit{\textup{Crit}}
\def\Hol{\textup{Hol}}

\def\Par{\textup{Par}}
\def\Hess{\textup{Hess}}
\def\dist{\textup{dist}}

\def\Lie{\textup{Lie}}

\def\Orb{\textup{Orb}}
\def\Img{\textup{Img}}

\def\fg{{\mathcal G}}

\def\uu{\textup{u}}

\def\Log{\textup{Log}}
\def\reg{\textup{reg}}

\begin{document}

\begin{abstract}  Assume $(X, \omega)$ is a compact symplectic manifold with a Hamiltonian compact Lie group action and the zero in the Lie algebra is a regular value of the moment map $\mu$.  
We prove that a finite energy symplectic vortex exponentially converges to (un)twisted sectors of the symplectic reduction at cylinder ends whose metrics grow up at least cylindrically fast, \emph{without assuming the group action on the level set $\mu^{-1}(0)$ is free}. It generalizes the corresponding results  by Ziltener \blue{\cite{ziltener-thesis, ziltener-exp}} under the free action  assumption. 

The result of this paper is the first step in setting up the quotient morphism moduli space induced by the authors in \cite{quotientmorphism}. Necessary preparations in understanding the structure of such moduli spaces are also introduced here.  The quotient morphism constructed in \cite{quotientmorphism} is a part of the project on the quantum Kirwan morphism by the authors (see \cite{L2, L2moduli, augL2}). 


\end{abstract}

\maketitle

\tableofcontents


\section{Introduction and main results}

Assume $(X, \omega)$ is a compact smooth symplectic manifold and $G$ is a connected compact Lie group whose Lie algebra is denoted by $\CG$. 
By a Hamiltonian $G$-action on $(X, \omega)$ with a moment map $\mu$, we mean that 
\begin{enumerate}
\item The Lie group $G$ acts on $X$ smoothly and we use $$
\ell:G\times X\to X, \quad (g, x)\mapsto \ell_g(x) := g  x
$$
to denote this action;
\item The moment map $\mu: X\to \CG$ is a proper smooth map and it is $G$-equivariant with respect to the $G$-action on $X$ and the adjoint action on $\CG$, i.e., 
$$
\mu(\ell_g(x))=\Ad_g(\mu(x)), \quad \text{ for any }g\in G, x\in X;
$$
\item The $G$-action is Hamiltonian in the sense that $\iota_{X_\xi}\omega=\langle d\mu(\cdot), \xi\rangle$, where $X_\xi$ denotes the infinitesimal action generated by $\xi\in \CG$ on $X$, i.e., 
$$
X_\xi(x):=\frac{d}{d\epsilon}|_{\epsilon=0}\ell_{\exp(\epsilon \xi)}(x),
$$ and $\langle\cdot, \cdot\rangle$ denotes the Killing inner product on the Lie algebra $\CG$. 
\end{enumerate}
It is clear from definition that associated to a Hamiltonian $G$-action on $(X, \omega)$, the moment map $\mu$ is not unique but can be shifted up to some element $\tau\in C(\CG)$, where $C(\CG)$ denotes the center of the Lie algebra $\CG$. Throughout this paper, we fix a moment map $\mu$ and assume $0\in \CG$ is a regular value. All results in this paper apply to any $\tau\in C(\CG)$ which is regular.  

The level set $\mu^{-1}(0)$ is compact and $G$-invariant. The assumption that $0$ is a regular value ensures that $\mu^{-1}(0)$ is a compact smooth submanifold of $X$ with codimension $\dim G$ and the $G$-action on it is locally free. Moreover, $(\mu^{-1}(0), \omega|_{\mu^{-1}(0)})$ is coisotropic, with $\ker\omega|_{\mu^{-1}(0)}$ the same as the subbundle spanned by vector fields generated by infinitesimal $\CG$-actions. It follows the quotient space $M:=\mu^{-1}(0)/G$ endows a symplectic orbifold structure, which is referred as the well-known Marsden-Weinstein symplectic reduction in literature.

To do the Hamiltonian Gromov--Witten or Floer theory for $(X, \omega)$ and its symplectic reductions, one needs to assign an almost complex structure $J$ on $X$ which is compatible with the Hamiltonian group action. In this paper, we choose and fix a such $J$ such that it is $G$-invariant and $\omega$-compatible (such almost complex structure always exists due to the compactness of the $G$ and the existence of compatible almost complex structure for any symplectic manifold). 
The almost complex structure $J$ then reduces to a compatible almost complex structure on the symplectic orbifold $M$.
From now on, we use the sextuple $(X, \omega, J, G, \mu, M)$ to denote the geometric structure described above, and use $(\bar\omega, \bar{J})$ to denote the reduced almost K\"ahler structure on the orbifold $M$. 

In general, on a symplectic orbifold and in particular now on the symplectic reduction $(M, \bar\omega)$, the corresponding orbifold Gromov--Witten theory was developed by Chen-Ruan, which is now referred as the quantum Chen-Ruan orbifold theory (see \cite{chen-ruan} and the references therein). Chen-Ruan orbifold theory is a theory for the inertia orbifold which includes the information also from twisted sectors besides the orbifold itself. Any possible moduli space theory adapted to a symplectic \emph{orbifold} must satisfy the property that any finite energy curve should converge to (un)twisted sectors at ends.  (The Fan-Jarvis-Ruan-Witten theory \cite{fan-jarvis-ruan} which is analogous to the orbifold Gromov--Witten theory is another strong evidence for the importance of considering the twisted sectors.)

In the paper \cite{L2},  the authors set up a new quantum cohomology ring structure for the symplectic reduction $(M, \bar{\omega})$ using the symplectic vortex equations attached to $(X, \omega, J, G, \mu, M)$. In particular, they prove that a finite energy $L^2$-symplectic vortex (i.e., a symplectic vortex with cylinder metric at each punctured end) (exponentially) converges to (un)twisted sectors at ends after some gauge transformations. This is an essential step in setting up the $L^2$-HGW (HGW is short for Hamiltonian Gromov--Witten) moduli space which is used to define a quantum cohomology field theory for the symplectic reduction $(M, \bar{\omega})$ (\cite{L2moduli}). 
In this paper, we generalize the results on the asymptotic convergence of a finite energy symplectic vortex to more general types of metrics which we name admissible metrics, i.e.,  the K\"ahler metrics on the Riemann surface which are of the form 
$$
e^{2b t}(dt^2+d\theta^2), \quad  b\geq 0 
$$
on each cylinder end $[0, \infty)\times S^1$. 

By a critical loop (see Definition \ref{def:critical}) of a constant connection, we mean a loop in $X\times \CG$ of the form
$$
x_\infty=(\exp(-\theta\eta_\infty)z_0, \eta_\infty), \quad z_0\in \mu^{-1}(0), \eta_\infty\in \CG, 
$$ 
which can be identified with a point in an (un)twisted sector of the symplectic orbifold $M$.  
We prove in this paper that

\begin{thm}\label{thm:main}
Assume $(P, w)=(P, u, A)$ is an admissible finite energy symplectic vortex over the punctured Riemann surface $(\dot\Sigma, j, h)$ with $k$-punctures. 
Then there exist some gauge transformation $\Phi\in \fg_P$,  $k$ critical loops $x_\infty^i=(z_\infty^i, \eta_\infty^i)$ of constant connections, $i=1, \cdots, k$,  and some constants $C>0$, $\delta_i>0$, $i=1, \cdots, k$, which only depend on $(X, \omega, G, \mu, J, M)$, such that 
$\Phi\cdot w=(\Phi^*u, \Phi^*A)$ is temporal and for any $\theta\in S^1$
\begin{enumerate}
\item $|d_{\Phi^*A} \Phi^*u(t, \theta)|\leq Ce^{-\delta t}$,
\item $|\mu(\Phi^*u)(t, \theta)|\leq Ce^{-\delta t-bt}$,
\item $|F_{\Phi^*A}(t, \theta)|\leq Ce^{-\delta t+bt}$,
\item $\|\dist(\Phi\cdot u(t, \cdot), z_\infty(\cdot))\|_{L^p(S^1)}+\|\Phi^*A(t, \cdot)-\eta_\infty(\cdot)\|_{L^p(S^1)}\leq Ce^{-(\delta+b(\frac{2}{p}-1))t}$,
for $p\geq 2$ and $\delta+b(\frac{2}{p}-1)>0$.
\end{enumerate}
Here we use $z_\infty, \eta_\infty, \delta, b$ to denote the ones with corresponding superscripts $i$'s, $i=1, \cdots, k$.

Moreover, such $\delta_i$'s can be taken as any positive number which is smaller than $\frac{1}{|\Hol_{\eta_i}|}$, where $|\Hol_{\eta_i}|$ is the order of the holonomy of the connection form $\eta_i$. 
\end{thm}

In particular, this proves the convergence of a finite energy symplectic vortex with admissible metric to the (un)twisted sectors of the symplectic reduction $M$ at punctured ends after some gauge transformation. The convergence is exponentially fast with some rate related to the first non-vanishing eigenvalue of certain (formal) Hessian operator (see Proposition \ref{prop:hess}).  
We refer readers to Proposition \ref{prop:hess},  Proposition \ref{prop:main} and Corollary \ref{cor:main} for more concrete and detailed results for Theorem \ref{thm:main}.

By using an approach which is different from Ziltener's \cite{ziltener-thesis, ziltener-exp} together with the study of the orbifold structure of symplectic reduction, from the technical point of view, Theorem \ref{thm:main} generalizes the following existed results for finite energy symplectic vortices:
\begin{enumerate}
\item It generalizes Ziltener's asymptotic convergence results in \cite{ziltener-thesis, ziltener-exp} by removing the 
assumption that the Lie group $G$-action on $\mu^{-1}(0)$ is free.  (We remark that a proof for affine vortices in compact K\"ahler manifolds or $\C^n$ was sketched by Venugopalan-Woodward in \cite[Section 4, 5]{venugopalan-woodward}. )
\item It generalizes the authors' results in \cite{L2} (which corresponds to $b=0$) by including other admissible metrics with $b>0$. In particular, the current result includes the affine metric, i.e., the case of $b=1$. 
\end{enumerate}
In fact, it is well-known to experts that one can expect certain exponential convergence of finite energy solutions from evolution type elliptic PDEs which are translation invariant, under the Morse--Bott condition. The translation invariant property corresponds to the cylindrical metric, i.e., $b=0$, here. For such case, usually one obtains a subsequence convergence first and then recover the $C^0$ exponential convergence with the help of the Morse--Bott condition. However, this method fails for non-translation invariant equations, which correspond to the case $b>0$. In this paper,  to deal with all admissible metrics, i.e., $b\geq 0$, we present a general strategy which is different from \cite{ ziltener-thesis, ziltener-exp}.  In particular in the proof we emphasize the only two ingredients in obtaining the exponential convergence: the a priori estimates and the Morse--Bott condition. This framework can be also applied to other evolution type elliptic PDEs than the one considered in this paper, though the interplay of these two ingredients might need to be modified.

From the application point of view, Theorem \ref{thm:main} is a basic step in constructing the quantum Kirwan morphism proposed by the authors (see \cite{L2, L2moduli, augL2}). Besides, the general setup in the proof, in particular the geometric understanding near the level set $\mu^{-1}(0)$, is necessary in constructing various moduli space theories related to symplectic reductions, including in the authors' work on the construction of quotient morphism moduli space in \cite{quotientmorphism}.

\begin{rem}[Historical remark on symplectic vortices in symplectic topology]
Using symplectic vortices to study the symplectic topology for symplectic manifolds admitting Hamiltonian Lie group action starts from the pioneering work by Cieliebak-Gaio-Salamon \cite{cieliebak-gaio-salamon} and independently by Mundet-i-Riera \cite{mundet1, mundet2} around 2000. 
In particular, a proposal on quantum Kirwan morphism was suggested by Gaio-Salamon in \cite{gaio-salamon} and then was developed by Ziltener \cite{ziltener-affine}, Woodward \cite{woodward1, woodward2, woodward3} in algebraic geometry setup and claimed recently by Tian-Xu in the expository article \cite{tian-xu} for K\"ahler manifolds with reductive group actions. (Other related work on symplectic vortices, e.g., on gauged Hamiltonian Floer theory and gauged Lagrangian Floer theory, see also \cite{frauenfelder1, frauenfelder2, xu-ham, wang-xu}.)

We remark that our approach in quantum Kirwan morphism proposed in \cite{L2, L2moduli, augL2} is essentially different from the approach above based on \cite{gaio-salamon}, though some technical results, including the current one, are recyclable. 
\end{rem}

The organization of the paper is as follows: 
\begin{itemize}
\item In Section \ref{sec:vortex}, we introduce the definition of finite energy symplectic vortices with admissible metrics. 
\item In Section \ref{sec:darboux}, we set-up a normal form near the regular level set $\mu^{-1}(0)$ using the equivaraint Darboux theorem.
\item In Section \ref{sec:critical}, we explain the structure of the space of critical loops and its relation to (un)twisted sectors of the orbifold $M$.
\item In Section \ref{sec:functional}, based on the normal form induced in Section \ref{sec:darboux}, we introduce an action function for a fixed map and derive an isoperimetric type inequality for it. 
Then we relate this action function to the Yang-Mills-Higgs energy for a symplectic vortex. After these preparations, one is able to obtain the exponential energy density decay for finite energy symplectic vortices with admissible metrics. The derivation is included in Section \ref{sec:thm}. 
\item In Section \ref{sec:thm}, we state and prove another useful interpretation for the asymptotic convergence result as Proposition \ref{prop:main}, and then give the proof of Theorem \ref{thm:main}.
\end{itemize}

The conventions we use in this paper are as follows:
\begin{itemize}
\item  For the Lie group $G$, its Lie algebra structure on $\CG=T_e G$ is defined via the Lie bracket of left invariant vector fields on $G$. Any $G$-principal bundle is a right principal bundle, so the action of $G$ is from the right. 
 \item The group $G$-action on $X$ is  from the left, denoted by $\ell_g$.  The tangent map of the $G$-action on $X$ is denoted by $\ell_{g*}: T_xX\to T_{gx}X$. 
\item We use $\exp: \CG\to G$ to denote the exponential map for the Lie group $G$. The infinitesimal vector field generated by $\xi\in \CG$ is defined as 
$$X_{\xi}(x):= \frac{d}{d\epsilon} \Big|_{\epsilon=0}\exp(\epsilon\xi)x.$$ 
Note that $[X_{\xi}(x), X_{\eta}(x)] = - X_{[\xi,\eta]}(x)$ as $X_{\xi}(x)$ is defined by the right invariant field associated to $\xi\in \CG$. 
\item For $\Phi\in \Aut(P)$, where $\Aut(P)$ is the gauge transformation group of the $G$-principal bundle $P$, it acts on any pair of a $G$-equivariant map $u_G: P\to X$ and a connection $A$ via pulling back, i.e.,  
$$
\Phi\cdot(u_G, A)=(\Phi^*u_G, \Phi^*A). 
$$
In particular, when $P\cong D\times G$ is trivial, here $D$ could be a surface or $S^1$ in this paper, $\Phi$ will be identified with a map $g: D\to G$ and the $(u_G, A)$ pair will be identified with the pair $(u, \eta)\in C^{\infty}(D, X)\times \Omega^1(D, \CG)$. Then  
$$
g\cdot (u, \eta)=(\ell_{g^{-1}}(u), Ad_{g^{-1}} \eta+g^{-1}dg).
$$
The Lie algebra of $\Aut(P)$ is $\Omega^0(P^{\Ad})$, the section of the associated bundle for the adjoint action $\Ad: G \to GL(\CG)$. 
The infinitesimal action of $\xi \in \Omega^0(P^{\Ad})$ at $(u, \eta)$ is given by $( -X_{\xi}, d\xi + [\eta, \xi])$.
\end{itemize}

\section{The finite energy symplectic vortices with admissible metrics}\label{sec:vortex}

\subsection{The symplectic vortices with admissible metrics}
Assume $(\Sigma, j)$ is a genus $g_\Sigma$ smooth closed Riemann surface with $k$ ordered marked points. Denote by $(\dot\Sigma, j)$ the punctured one by removing all marked points. 
\begin{defn}A K\"ahler metric $h$ on $(\dot\Sigma, j)$ is called \emph{admissible}, if for each marked point $p_i$ in $\Sigma$, $i=1, \cdots, k$, there exists a neighborhood $U_i$ of $p_i$ in $\Sigma$ such that the K\"ahler structure of the end $(E_i:=U_i\setminus\{p_i\}, j, h|_{E_i})$ is isometric to $([0, \infty)\times S^1, j_0, e^{2b_i t}(dt^2+d\theta^2))$ with $b_i\geq 0$, $i=1, \cdots, k$.
Here $j_0$ denotes the standard complex structure on cylinder, i.e., $j_0(\frac{\del}{\del t})=\frac{\del}{\del\theta}$, $j_0(\frac{\del}{\del \theta})=-\frac{\del}{\del t}$ with $t\in [0, \infty), \theta\in S^1$.
\end{defn}

Assume $P$ is a principal $G$-bundle over $\dot\Sigma$. 
Denote by $C^\infty_G(P, X)$ the set of smooth $G$-equivariant maps from $P$ to $X$, and by $\CA(P)$ the set of connection $1$-forms on $P$.
Denote by $Y=P\times_GX$ the associated bundle of $X$ over $\dot{\Sigma}$, and by $P^\Ad:=P\times_G\CG$ the adjoint bundle.  
Then one can identify $C^\infty_G(P, X)$ with the set of smooth sections of $Y\to \dot{\Sigma}$,  and identify $\CA(P)$ with $\Omega^1(\dot{\Sigma}, P^\Ad)$.

A connection $A\in \CA(P)$ induces a splitting of the tangent bundle of $Y$ as 
$$
TY=H_AP\oplus (P\times_GTX), 
$$
with $H_AP:=\ker A$ as the horizontal distribution of $TP$. Here the $G$-action on $TX$ is induced from the $G$-action on $X$. Denote by $\Pi_A$ the endomorphism of $TY$ induced by the projection $\pi_A: TY\to P\times_GTX$ from this splitting. 

The vector bundle $P\times_GTX\to \dot{\Sigma}$ carries the  complex structure induced from the almost complex structure $J$, since $J$ is assumed to be $G$-invariant. We still use $J$ to denote this complex structure. Together with the horizontal lifting $j_A$ on $H_AP$ from the complex structure $j$ on $\dot{\Sigma}$, $Y$ carries an almost complex structure defined as 
$$
J_A:=j_A\oplus J. 
$$

For any map $u_G\in C^\infty_G(P, X)$, denote by $u$ the corresponding section of $Y\to \dot{\Sigma}$. The $A$-twisted tangent map of $u_G$ is defined as the $u^*TY$-valued $1$-form over $\dot{\Sigma}$
$$
d_Au:=\Pi_A\circ du.
$$
Denote by $\delbar_A u$ the $(0, 1)$-part of $d_Au$ with respect to the pair of complex structures $(j, J_A)$. 
In the context below, we do not distinguish $u_G$ and $u$ and readers should not be confused.  

\begin{rem} For a (local) trivialization $P|_U=U\times G$ with holomorphic coordinates $(t, s)$ for $U$, locally  
$$
d_Au(\frac{\del}{\del t})=\frac{\del u}{\del t}+X_{A(\frac{\del}{\del t})}, \quad 
d_Au(\frac{\del}{\del s})=\frac{\del u}{\del s}+X_{A(\frac{\del}{\del s})}
$$
and 
\beastar
\del_A u&=&\frac{1}{2}(d_Au-J\circ d_A u\circ j)\\
&=&\frac{1}{2}[(\frac{\del u}{\del t}+X_{A(\frac{\del}{\del t})})-J(\frac{\del u}{\del s}+X_{A(\frac{\del}{\del s})})]dt+\frac{1}{2}J[(\frac{\del u}{\del t}+X_{A(\frac{\del}{\del t})})-J(\frac{\del u}{\del s}+X_{A(\frac{\del}{\del s})})]ds\\
\delbar_A u&=&\frac{1}{2}(d_Au+J\circ d_A u\circ j)\\
&=&\frac{1}{2}[(\frac{\del u}{\del t}+X_{A(\frac{\del}{\del t})})+J(\frac{\del u}{\del s}+X_{A(\frac{\del}{\del s})})]dt+\frac{1}{2}J[(\frac{\del u}{\del t}+X_{A(\frac{\del}{\del t})})+J(\frac{\del u}{\del s}+X_{A(\frac{\del}{\del s})})]ds.
\eeastar
\end{rem}
\begin{defn}
A symplectic vortex over $(\dot\Sigma, j, h)$ is a smooth principal $G$-bundle $P$ over $\dot\Sigma$ together with a 
pair of maps $w:=(u, A)\in C^\infty_G(P, X)\times \CA(P)$ which satisfies the following nonlinear PDEs
\bea
\delbar_{A}u&=&0\\
F_A+*_h\mu(u)&=&0, 
\eea 
where $*_h$ is the Hodge $*$ operator with respect to the K\"ahler metric $h$. 

If moreover $h$ is admissible, we say $(P, w)$ is a \emph{symplectic vortex with admissible metric}. 

\end{defn}

\subsection{The Yang-Mills-Higgs (YMH) energy}

The Yang-Mills-Higgs (YMH) energy is a functional defined for pairs $(P, w)$ with $P$ a $G$-principal bundle over $\dot\Sigma$ and $w=(u, A)\in C^\infty_G(P, X)\times \CA(P)$ as
\bea\label{eq:e}
E(P, u, A)&:=&\int_{\dot{\Sigma}}e(P, u, A)\,\nu_{h}\nonumber\\
e(P, u, A)&:=&\frac{1}{2}(|d_Au|_h^2+|F_A|_h^2+|\mu\circ u|_h^2),
\eea
where $\nu_h$ is the volume form of $(\dot{\Sigma}, h)$, and $|\cdot|_h$ denotes the norm induced by the K\"ahler metric $h$ and the corresponding target metrics for (vector valued) forms over $\dot{\Sigma}$. We omit YMH and only call it the energy if no confusion could happen. 
The function $e(P, u, A): \dot{\Sigma}\to [0, \infty)$ is called the (YMH) energy density.  

The equality 
\be\label{eq:energyeq}
e(P, u, A)\nu_h=(|\delbar_A u|_h^2+\frac{1}{2}|F_A+*_h\mu(u)|_h^2)\nu_h+(u^*\omega-d\langle\mu(u), A\rangle)
\ee
characterizes the key feature of the YMH energy density and in particular indicates that the YMH energy is the right energy in studying symplectic vortices. To be more concrete, 
\begin{enumerate}
\item The first term $(|\delbar_A u|_h^2+\frac{1}{2}|F_A+*_h\mu(u)|_h^2)\nu_h$ in \eqref{eq:energyeq} vanishes if and only if $(P, u, A)$ is a symplectic vortex over $(\dot{\Sigma}, j, h)$;
\item The integral of the second term $u^*\omega-d\langle\mu(u), A\rangle$ in \eqref{eq:energyeq} is a  topological invariant, whenever  the pair $(P, w)$ have a nice asymptotic behavior (e.g., as the asymptotic convergence proved in this paper). 
\end{enumerate}

The derivation of \eqref{eq:energyeq} which is given in \cite[Proposition 3.1]{cieliebak-gaio-salamon} is based on the following two identities whose proofs are straightforward. In particular, when take trivial connection $A=0$, this recovers the corresponding identities for symplectic manifolds.
\begin{lem} 
\begin{enumerate}
\item $|d_Au|_h^2=|\del_Au|_h^2+|\delbar_Au|_h^2+2\langle F_A, *_h(\mu\circ u)\rangle$;
\item $u^*\omega-d\langle \mu\circ u, A\rangle=\frac{1}{2}(|\del_Au|_h^2-|\delbar_Au|_h^2)\nu_h$.
\end{enumerate}
\end{lem}
By a \emph{finite energy symplectic vortex} over $(\dot\Sigma, j, h)$, we mean that a symplectic vortex $(P, w)$ whose YMH energy is finite.

\subsection{The gauge action}
From now on, we \emph{fix} a punctured Riemann surface $(\dot\Sigma, j)$ with a chosen K\"ahler metric $h$ as domain. Moreover, we \emph{fix} a principal $G$-bundle $P\to \dot\Sigma$, and so we omit $P$ from the pair $(P, w)$ if there is no danger of confusion.  
Denote by $\fg:=\fg_P$ the gauge transformation group of $P$, and denote by 
$$
\CB:=\CB_P:=C^\infty_G(P, X)\times \CA(P).
$$
The gauge action on $P$ induces the $\fg$-action on $\CB$ as 
$$
\Phi\cdot w=(\Phi^*u, \Phi^* A).
$$ 
Regarding this gauge action, it is clear that 
\begin{lem}For any $w\in \CB$ and any $\Phi\in \fg$,
\begin{enumerate}
\item $E(\Phi\cdot w)=E(w)$;
\item $\Phi\cdot w$ is a symplectic vortex if $w$ is a symplectic vortex .
\end{enumerate}
\end{lem}

Assume $(P, u, A)$ is a symplectic vortex over $(\dot\Sigma, j, h)$. Over each end $E_i$ identified with $[0, \infty)\times S^1$, $i=1, \cdots, k$, where $k$ is the number of punctures, we can take a global trivialization of $P|_{E_i}\to E_i$ and after this trivialization, the connection on $P|_{E_i}$ can be identified with some 
$A_i=A_{i,1}(t, \theta)\,dt+A_{i, 2}(t, \theta)\,d\theta$, $(t, \theta)\in [0, \infty)\times S^1$, where  
$$
A_{i, 1}, A_{i, 2}: [0, \infty)\times S^1\to \CG.
$$
It is clear from the ODE theory that one can always find a (unique) smooth gauge transformations $\Phi_i\in \fg_{P|_{E_i}}\cong C^{\infty}([0, \infty)\times S^1, G)$ and some $\eta_i: [0, \infty)\times S^1\to \CG$, such that
$$
\Phi_i^*A_i=\eta_i(t, \theta)\,d\theta, \quad \Phi_i(0, \cdot)\equiv e. 
$$
Such connection of the form $\eta(t, \theta)\,d\theta$ is called a temporal connection over $[0, \infty)\times S^1$. 
Obviously, one can always find a gauge transformation $\Phi\in \fg_P$ whose restriction on each end is $\Phi_i$. Then for every connection $A\in \CA(P)$, there exists some gauge transformation to make $A$ be temporal at each end. 
In particular, when $w=(u, A)$ is a symplectic vortex and $A$ is temporal at each end, we call 
$w$ a \emph{temporal symplectic vortex}.

\section{The normal form near $\mu^{-1}(0)$}\label{sec:darboux}

In this section,  we describe a normal form for the symplectic manifold $(X, \omega)$ that admits the Hamiltonian $G$-action with moment map $\mu$ near a small $G$-invariant tubular neighborhood of $\mu^{-1}(0)$. The 
construction will be used for a better understanding of the formal Hessian operator defined in Section \ref{sec:critical} as well as in the later definition of the local functional in Section \ref{sec:functional}. 

For every $x\in X$, we denote by $L_x: \CG\to T_xX$ the infinitesimal action of the Lie algebra $\CG$, i.e., 
$$
L_x(\xi):=X_\xi(x)
$$
to emphasize the role as a linear operator.

The assumption that $0$ is a regular value of the moment map leads to the following well-known decomposition of $T_z\mu^{-1}(0)$ at any $z\in \mu^{-1}(0)$. We summarize them as in the following lemma and also  fix some notations.  
\begin{lem}Assume $0$ is a regular value of the moment map $\mu$. Then 
for any $z\in \mu^{-1}(0)$, $L_z$ is injective and 
$$
T_z\mu^{-1}(0)=\ker d_z\mu=(\Img(JL_z))^\perp,
$$ 
where $\perp$ denotes the orthogonal complement with respect to the metric $\omega(\cdot, J\cdot)$ for $X$.
Further,  $T_z\mu^{-1}(0)$ has the decomposition 
\be\label{eq:decomp}
T_z\mu^{-1}(0)=\Img(L_z)\oplus H_z,
\ee
where $H_z:=\ker(d_z\mu\circ J)=(\Img(L_z))^\perp\cap T_z\mu^{-1}(0)$. The splitting \eqref{eq:decomp} is $G$-invariant. 
\end{lem}
We use $\pr_G: T_z\mu^{-1}(0)\to \Img(L_z)$ to denote the projection from the decomposition \eqref{eq:decomp}.

Now we consider the vector bundle 
$$
N:=\Img(JL)\to \mu^{-1}(0),
$$ 
and denote by $N^\epsilon$ a tubular neighborhood of the zero section, where $0<\epsilon<\inf_{z\in \mu^{-1}(0)}i_z$ with $i_z$ the injective radius of the exponential map defined using the metric $\omega(\cdot, J\cdot)$. Here, $\inf_{z\in \mu^{-1}(0)}i_z>0$ due to the properness of moment map $\mu$. 

It follows that the exponential map $\exp^{\nabla}$ induced from the Levi-Civita connection defines a local diffeomorphism 
$$
\exp^\nabla: N^{\epsilon}\to X, \quad JL_z(\xi)\mapsto \exp^\nabla_z(JL_z(\xi)).
$$ 
We use $\CU^\epsilon\subset X$ to denote its image. In particular the image of the zero section is $\mu^{-1}(0)\subset \CU^\epsilon$.
Notice that the pull-back $G$-action on $X$ behaves as 
\be\label{eq:action}
g(z, JL_z(\xi))=(gz, JL_{gz}(\Ad_g(\xi)))
\ee
due to the $G$-variance of the almost K\"ahler structure. 
By suitable shrinking of $\epsilon$, we can further assume $\CU^\epsilon$ (so is $N^\epsilon$) is $G$-invariant. 

By the pulling back of the exponential map, we identify $(\CU^\epsilon:=\exp^\nabla(N^\epsilon), \omega|_{\CU^\epsilon}, G, \mu|_{\CU^\epsilon}, J)$ with 
$$(N^\epsilon, (\exp^\nabla)^*(\omega|_{\CU^\epsilon}), G, (\exp^\nabla)^*(\mu|_{\CU^\epsilon}), (\exp^\nabla)^*J).$$
To simplify notations, we omit the notation of pull-backs by $\exp^\nabla$ and use $(N^\epsilon, \omega, G, \mu, J)$ to denote the corresponding structures on the tubular neighborhood $N^\epsilon$.

Now we construct a symplectic form on $N^\epsilon$ as follows. 
Denote by $\CF_G$ the foliation on $\mu^{-1}(0)$  defined by the vector field $\Img(L)$, and notice that there is a natural bundle map 
$$
\iota_G: \Img(JL)\to T^*\CF_G, \quad \iota_G(z, JL_z(\xi))=\iota_{JL(\xi)}\omega|_z:=(z, \omega(JL_z(\xi), \cdot)).
$$
We use it to pull back the canonical one-form $\Theta_G$ on $T^*\CF_G$ and then obtain 
$$
\Omega_G:=\iota_G^*(d\Theta_G)\in \Omega^2(\Img(JL)).
$$
Define a closed two-form on $N^\epsilon$ as 
$$\omega_0:=\pi^*\omega|_{\mu^{-1}(0)}+\Omega_G.$$
From the construction, we have the following lemma.
\begin{lem}\begin{enumerate}
\item $\omega_0|_{\mu^{-1}(0)}=\omega|_{\mu^{-1}(0)}$, hence $\omega_0$ is a symplectic form whenever $\epsilon$ is a small enough;
\item The $G$-action on $N^\epsilon$ is Hamiltonian with respect to $\omega_0$, whose moment map $\mu_0$ satisfies
$\mu_0(JL_z(\xi))=\xi$.
\end{enumerate}
\end{lem}
We call $(N^\epsilon, \omega_0, G, \mu_0)$ the normal form near the level set $\mu^{-1}(0)$. 
Regarding it, we have the following equivariant Weinstein-Darboux theorem.
\begin{prop}\label{prop:darboux}
For sufficiently small $\epsilon>0$, there exists a $G$-equivariant diffeomorphism
$$
\phi: N^\epsilon\to N^\epsilon,
$$ such that 
\bea
&&\phi^*\omega=\omega_0,\quad \phi^*\mu=\mu_0, \quad\phi|_{\mu^{-1}(0)}=\id_{\mu^{-1}(0)},\label{eq:darboux1}\\
&&d\phi|_{TN^\epsilon |_{\mu^{-1}(0)}}=\id|_{TN^\epsilon |_{\mu^{-1}(0)}}.\label{eq:darboux2}
\eea
\end{prop}
\begin{proof}Equalities in \eqref{eq:darboux1} immediately follow from the standard equivariant Darboux theorem (e.g., see \cite{weinstein}) and only \eqref{eq:darboux2} is based on the current symplectic reduction condition whose proof we are giving now. 

To show \eqref{eq:darboux2}, we take $\{\xi_i\}_{i=1, \cdots, \dim(G)}$ as a chosen basis for $\CG$. Because $\phi^*\mu=\mu_0$ on $\mu^{-1}(0)$, we can find smooth functions 
$f_i: \CU^\epsilon\to \R$ for $i=1, \cdots, \dim(G)$ with $f_i|_{\mu^{-1}(0)}\equiv 1$, such that 
$$
\langle\phi^*\mu, \xi_i\rangle=f_i\langle\mu_0, \xi_i\rangle, \quad i=1, \cdots, \dim(G).
$$
Then by taking differential, we obtain
\beastar
\langle\phi^*d\mu, \xi_i\rangle=d\langle\phi^*\mu, \xi_i\rangle=\langle\mu_0, \xi_i\rangle df_i+f_i \langle d\mu_0, \xi_i\rangle.
\eeastar
When restricted to $\mu^{-1}(0)=\mu^{-1}_0(0)$, the first term of right hand side vanishes, and the second term of right hand side becomes $ \langle d\mu_0, \xi_i\rangle$. Hence it follows 
$$
\phi^*d\mu|_{\mu^{-1}(0)}=d\mu_0|_{\mu^{-1}(0)}, 
$$
and then 
$d\phi|_{TN^\epsilon |_{\mu^{-1}(0)}}=\id|_{TN^\epsilon |_{\mu^{-1}(0) }}$ since $d_z\mu$ is an isomorphism when restricted to $\Img(JL)$.
\end{proof}
Proposition \ref{prop:darboux} guarantees us to assume the target space as
$$
(N^\epsilon, \omega_0, G, \phi^*J, \mu_0)
$$
by looking at the map $(\phi^*u, A)$, whenever $u$ falls into the tubular neighborhood $N^\epsilon$ of $\mu^{-1}(0)$. In particular, we notice \eqref{eq:darboux2} ensures that $\phi^*J=J$ on $\mu^{-1}(0)$.
 Later in Section \ref{sec:functional}, when we work under such normal form, we assume everything is pulled back via some $\phi$ and omit such $\phi$.

\section{The gauge theory over $S^1$}\label{sec:critical}
\subsection{The gauged loops and the holonomy}
We fix a principal $G$-bundle $P_{S^1}$ over $S^1$. It is always trivial since $G$ is connected. Denote by 
$\CB_{S^1}:=C_G^\infty(P_{S^1}, X)\times \CA(P_{S^1})$, which can be identified with the space of loops in $X\times \CG$ after fixing a trivialization. Denote by $\fg_{S^1}$ the gauge transformation group of $P_{S^1}$, which can be identified with the loop space of $G$. We use $LX, LG, L\CG$ to denote loop spaces of $X, G, \CG$ respectively, and use $L_0G$ to denote the based gauge group, i.e, the normal subgroup of $LG$ whose elements satisfy $g(0)=e$. 

The gauge group $LG$ acts on $LX\times L\CG$ as 
$$
g\cdot y=g\cdot (x, \eta)=(g^{-1}x, \Ad_{g^{-1}}\eta+g^{-1}\dot{g}).
$$
 For a fixed $g\in LG$, the tangent map 
is 
$$
dg|_y(v, \xi)=( (\ell_{g^{-1}})_*(v), \Ad_{g^{-1}}\xi),
$$
for $v\in C^\infty(S^1, z^*TX)$, $\xi\in L\CG$.

The infinitesimal action of the Lie algebra of $LG$, which is the same as $T_e(LG)=L\CG$, is 
\bea\label{eq:infgauge}
\frac{d}{d\epsilon}\Big|_{\epsilon=0}(\exp(\epsilon\tau)\cdot y) 
=(-X_\tau(z), \dot{\tau}+[\eta, \tau]).
\eea
Here $\dot{\tau}$ is the short notation for $\dfrac{d\tau}{d\theta}$. 
Clearly the based gauge $L_0G$ acts freely on $L\CG$ (hence freely on $LX\times L\CG$) by the uniqueness result for ODEs. Moreover, it is well-known that orbit space is in fact $G$ and the quotient map maps every connection to its holonomy. Now for later use, we make this construction explicit. 

Consider the horizontal path $\Psi:[0, 2\pi]\to G$ as
\be\label{eq:psi}
d_\eta\Psi=\Psi'+\eta\Psi=0, \quad \Psi(0)=e.
\ee
It is a linear ODE thus has a unique solution which we denote by $\Psi_\eta:[0, 2\pi]\to G$. The group element $\Psi_\eta(1)=:\Hol_\eta$ is the holonomy of the connection $\eta$. 
Notice that $\Psi_{g\cdot\eta}=g^{-1}\Psi_\eta$, and in particular this indicates the holonomy is $L_0G$-invariant.

Since we assume that $G$ is connected and compact,  the exponential map of the Lie group $G$ is surjective. We can choose 
$\log\Hol_\eta \in \CG$ such that $\exp(\log\Hol_\eta) = \Hol_\eta$.  Define a based gauge transformation  as
\be\label{eq:h}
h_\eta(\theta):=\Psi_{\eta}(\theta)\exp(-\theta\log\Hol_\eta)\in L_0G.
\ee
It follows
\beastar
h'_\eta+\eta h_\eta-h_\eta\eta&=&0\\
h_{g\cdot\eta}(\theta)&=&g^{-1}(\theta)h_\eta(\theta)\\
h_\eta\cdot\eta&=&-\log(\Hol_\eta).
\eeastar
In particular,  it maps $\eta$ to the constant connection $-\log(\Hol_\eta)$ with the same holonomy. Notice that this constant connection may not be unique.

The following remark will be used later in Section \ref{sec:functional}. 
\begin{rem}Considering the principal $L_0G$-bundle $L\CG$ over $G$, for any smooth path $h: (a, b)\to G$, we can take a smooth path of \emph{constant} loops in $L\CG$ as $\eta: (a, b)\to \CG$ such that 
$$
\Hol_{\eta(t)}=h(t). 
$$
We use $\Log(h)$ to denote a taken choice though the choice is not canonical. 
\end{rem}

Next we calculate the slice of the full gauge $LG$-action on $\CB_{S^1}=LX\times L\CG$. Since both the metric on $X$ and the Killing metric on the Lie group are $G$-invariant, it follows the induced $L^2$-metric on $\CB_{S^1}$ defined as
$$
\|(v, \xi)\|^2_{L^2}=\int_{S^1}(|v|^2 + |\xi|^2)\,d\theta
$$
for $y = (x, \eta)\in \CB_{S^1}$ and $(v, \xi)\in T_y\CB_{S^1}=\Gamma(S^1, x^*TX)\times L\CG$,  is gauge $LG$-invariant.

Denote the $L^2$-decomposition as
$$
T_y\CB_{S^1}=\CV_y\oplus  T_y(\Orb_{\fg_{S^1}}(y)), 
$$ 
where $\Orb_{\fg_{S^1}}(y)$ denotes the gauge orbit at $y=(x, \eta)\in \CB_{S^1}$ and $\CV_y$ is the
$L^2$-orthogonal complement of $T_y(\fg_{S^1}\cdot y)$. Denote by $\Pi_{\CV_y}$ the $\CV_y$-projection. 

\begin{lem}\label{lem:isotropy}
  For any $y=(x, \eta)\in \CB_{S^1}$, then 
$$\CV_y=\{(v, \xi)\in \Gamma(S^1, x^*TX)\times L\CG|d\mu\circ J(v)+\dot{\xi}+[\eta, \xi]=0\}$$ 
and it is $\fg_{S^1}(y)$-invariant, where $\fg_{S^1}(y)$ denotes the isotropy group (stabilizer) at $y$ which is isomorphic to a subgroup of $G_{x(0)}$. Hence the isotropy group $\fg_{S^1}(y)$ is finite whenever $G_{x(0)}$ is finite. 
\end{lem}

\begin{proof} Take any $y=(x, \eta)\in \CB_{S^1}$, we have shown that the infinitesimal action of $\tau\in \Lie(\fg_{S^1})$ on $y$ is 
$(-X_\tau(x), \dot{\tau}+[\eta, \tau])$. 
It follows that $(v, \xi)\in \CV_y$ if and only if 
$$
\int_{S^1}(\omega(-X_\tau(x), Jv)+<\dot{\tau}+[\eta, \tau], \xi>)d\theta=0
$$
for any $\tau\in \Lie(\fg_{S^1})$.
Using the $G$-invariance property of the metrics, this is further equivalent to 
\beastar
\int_{S^1}<d\mu\circ J(v)+\dot{\xi}+[\eta, \xi], \tau>d\theta=0
\eeastar
for any $\tau\in \Lie(\fg_{S^1})$, and hence
$$
d\mu\circ J(v)+\dot{\xi}+[\eta, \xi]=0.
$$
Since the $L^2$ metric  on $\CB_{S^1}$  is $\fg_{S^1}$-invariant, it follows 
$\CV\to \Orb_{\fg_{S^1}}(y)$ is a $\fg_{S^1}$-equivariant bundle and $\CV_y$ is $\fg_{S^1}(y)$-invariant.   

To calculate  the isotropy group $\fg_{S^1}(y)$, we notice that for any 
$g\in LG$, $g$  fixes $y=(x, \eta)$ if and only if 
$$
g(s)\in G_{x(s)}, \quad g'+\eta g-g\eta=0, \quad s\in [0, 1].
$$
By the uniqueness result of ODEs with initial conditions, 
the group homomorphism 
$$
\fg_{S^1}(y)\mapsto G_{x(0)}, \quad g\mapsto g(0)
$$
is injective. Hence $\fg_{S^1}(y)$ is a subgroup of $G_{x(0)}$. 
Here this $0\in S^1$ can be replaced to any $\theta\in S^1$. 

 \end{proof}

Near the regular level set $\mu^{-1}(0)$, one can find some constant $\epsilon_{\reg}>0$, such that the $\epsilon_{\reg}$ neighborhood of $\mu^{-1}(0)$ is $G$-invariant and the $G$-action on it is locally free. Then it follows 
$$\CB_{S^1, \epsilon_\reg}:=\{y=(x, \eta)\in LX\times L\CG|\dist(x, \mu^{-1}(0))<\epsilon_{\reg}\}$$
is  $LG$-gauge invariant, and the $LG$-gauge action is locally free from the above lemma.  

Moreover, we can make $\epsilon_{\reg}$ small enough so that the isotropy groups of $LG$ on $\CB_{S^1, \epsilon_{\reg}}$ one-to-one correspond to (a subset of) conjugacy classes  $\{\mu^{-1}(0)_{(H)}/G| (H)\in \Lambda\}$, where $\Lambda$ denotes the set of conjugacy classes of isotropy groups of points in $\mu^{-1}(0)$ by $G$-action. Notice that the compactness of $G$ and $\mu^{-1}(0)$ ensures that $\Lambda$ is a finite set. 

For each $x\in \mu^{-1}(0)$, define  
$$
d_{x, G}:=\inf_{k, k'\in G_x}\dist_G(k, k'),
$$
where $\dist_G$ denotes the distance on $G$ induced from the Killing metric. 
Since the Killing metric on $G$ is adjoint invariant, $d_x$ is constant on each conjugacy class $(G_x)$. Due to the finiteness of orbit types from the discussions above, the number 
\be\label{eq:Gdist}
d_{\mu^{-1}(0), G}:=\inf_{x\in \mu^{-1}(0)}d_{x, G}>0. 
\ee
We are going to use this property in Section \ref{sec:functional}.

\subsection{The critical loops}

\begin{defn}\label{def:critical}A loop $y=(x, \eta)\in LX\times L\CG$ is called a \emph{critical loop}, if it satisfies the following two equations
$$
\mu\circ x\equiv 0, \quad \dot{x}+X_\eta(x)=0.
$$
Denote by $\Crit\subset LX\times L\CG$ the set of all critical loops. 
\end{defn}

Denote by $\CE_{S^1}\to \CB_{S^1}$ a vector bundle with fiber at $y=(x, \eta)\in LX\times LG $ as 
$$
\CE_{S^1}|_y=\Gamma(S^1, x^*TX\oplus \CG).
$$
Define a section
$$
\widetilde\Upsilon(y):=(\Upsilon(y), \mu(x)):=(\dot{x}+X_{\eta}(x), \mu(x)).
$$
Then by definition, $\Crit=\widetilde{\Upsilon}^{-1}(0)$.

\begin{lem}\label{lem:da}
The section $\widetilde\Upsilon$ is gauge equivariant, i.e., for any $g\in LG$ and $y=(x, \eta)$
$$
\widetilde\Upsilon(g\cdot y)= ((\ell_{g_{-1}})_* (\Upsilon(y)), \Ad_{g^{-1}}\mu).
$$
It follows then $\Crit$ is $LG$-invariant.
\end{lem}
\begin{proof}
Since $\mu$ is $G$-equivariant, $\mu(g^{-1}x)=\Ad_{g^{-1}}\mu(x)$.
By a direct computation, 
\beastar
\Upsilon(g\cdot y)&=&\frac{d}{d\theta}(g^{-1}x)+X_{g^{-1}\dot{g}+\Ad_g\eta}(g^{-1}x)\\
&=&[(\ell_{g_{-1}})_*(\dot{x})-X_{g^{-1}\dot{g}}(g^{-1}x)]+X_{g^{-1}\dot{g}+\Ad_g\eta}(g^{-1}x)\\
&=&(\ell_{g_{-1}})_* (\dot{x}+X_{\eta}(x))\\
&=&(\ell_{g_{-1}})_* (\Upsilon(y)).
\eeastar

\end{proof}

\begin{lem}
If $y=(x, \eta)\in \Crit$, then the isotropy group $ \fg_{S^1}(y) \cong C_{G_{x(0)}}(\Hol_\eta), 
$ the centralizer of $\Hol_\eta$ in $G_{x(0)}$. In particular, it is finite and the isomorphism  is given by 
$$
C_{G_{z(0)}}(\Hol_\eta)\to \fg_{S^1}(y), \quad g_0\mapsto g(\theta)=h_\eta(\theta)g_0h^{-1}_\eta(\theta),
$$
where $h_\eta$ is defined as in \eqref{eq:h}. 
\end{lem}
\begin{proof} From the proof of Lemma  \ref{lem:isotropy},   we notice  that  for $g\in \fg_{S^1}(y) $ the equation $g'+\eta g-g\eta=0$ is equivalent to 
$$
\Ad_{\Psi^{-1}_\eta(s)}g(s)\equiv g(0)
$$
for $\Psi: [0, 2\pi] \to  G$ satisfying (\ref{eq:psi}). 
It follows $g(1)=\Psi_\eta(1)g(0)(\Psi_\eta(1))^{-1}$. Hence $g\in LG$ solves the equation $g'+\eta g-g\eta=0$ if and only if $$
\Phi_\eta(1)g(0)(\Psi_\eta(1))^{-1}=g(0), \quad \text{i.e., } \Ad_{\Hol_\eta}g(0)=g(0).
$$
On the other hand, notice that whenever $g\in LG$ fixes $\eta$, i.e., satisfies the equation $g'+\eta g-g\eta=0$, it automatically fixes $x$ when $(x, \eta)\in \Crit$.   Moreover,  given $g_0\in G_{x(0)}$, by a direct calculation, we get  
$$
g(\theta)=h_\eta(\theta)g_0h^{-1}_\eta(\theta)
$$
is a solution, which must be the unique solution.  
 \end{proof}

\begin{lem}\label{lem:linearization}Assume $y=(z, \eta)\in \Crit$ is a constant connection loop with $\eta\equiv \eta_0$. Then the linearization of the section $\widetilde{\Upsilon}$ at $y$ is
$$
D_y\widetilde{\Upsilon}(v, \xi)=(D_y\Upsilon(v, \xi), d_z\mu),
$$
with
\beastar
D_x\Upsilon(v, \xi)(\theta)&=&(\CL_{-X_{\eta_0}}v+X_{\xi(\theta)})(z(\theta))\\
&=&\frac{d}{ds}|_{s=0}\varphi_{-s *}v(z(\theta+s))+X_{\xi(\theta)})(z(\theta)),
\eeastar
for any $v\in C^\infty(S^1, z^*TX)=T_z(LX)$, $\xi\in L\CG=T_{\eta}\CG$, $\theta\in S^1$. Here $\varphi$ denotes the flow of the vector field $-X_{\eta_0}$. 

In particular, when $\eta_0=0$, 
$$
D_y\Upsilon(v, \xi)(\theta)=\dot{v}(\theta)+X_{\xi(\theta)}(z(\theta)). 
$$
\end{lem}
\begin{proof}Clearly, $D_y\widetilde{\Upsilon}=(D_y\Upsilon, d_z\mu)$, and we now calculate $D_y\Upsilon$. 
We use a torsion free affine connection $\nabla$ on $X$ to do the calculation. The result turns out to be independent of connections when $y$ is a critical loop with constant connection.  

For any $v\in C^\infty(S^1, z^*TX)=T_z(LX)$, $\xi\in L\CG=T_{\eta}\CG$, 
\beastar
D_y\Upsilon(v, \xi)(\theta)&=&\frac{d}{d\epsilon}|_{\epsilon=0}\Par^{-1}_{z(\theta), z_\epsilon(\theta)}(\Upsilon(z_\epsilon, \eta_\epsilon)(\theta))\in T_{x(\theta)}X, \quad \text{for any } \theta\in S^1, 
\eeastar
where 
$$
z_\epsilon(\theta):=\exp^\nabla_{z(\theta)}(\epsilon v(\theta)),\quad 
\eta_\epsilon(\theta):=\eta(\theta)+\epsilon\xi(\theta),$$ 
and the parallel transport $\Par$ and the exponential map $\exp^\nabla$ are defined by the affine connection $\nabla$. Further $D_y\Upsilon(v, \xi)(\theta)$ contains two terms from the expression of $\Upsilon$, and they can be calculated using the torsion free property of the connection as follows. 
\beastar
\text{The first term}&=&\frac{d}{d\epsilon}|_{\epsilon=0}\Par^{-1}_{z(\theta), z_\epsilon(\theta)}(\frac{d}{d\theta}\exp_{z(\theta)}(\epsilon v(\theta)))=\frac{Dv}{d\theta},\\
\text{The second term}&=&\frac{d}{d\epsilon}|_{\epsilon=0}\Par^{-1}_{z(\theta), z_\epsilon(\theta)}(X_{\eta(\theta+\epsilon\xi(\theta))}(\exp_{z(\theta)}(\epsilon v(\theta)))\\
&=&-\frac{Dv}{d\theta}+X_{\xi(\theta)}(z(\theta))+[\widetilde{\frac{\del\Phi}{\del \epsilon}}, X_{\eta(\theta)}]|_{\epsilon=0}(z(\theta)),
\eeastar
where $\frac{Dv}{d\theta}$ denotes the covariant derivative of the vector field $v$ along the curve $z$, and 
$\widetilde{\frac{\del\Phi}{\del \epsilon}}$ denotes the local extension of the vector field $\frac{\del\Phi}{\del \epsilon}$ for 
$$
\Phi(\epsilon, \theta):=\exp^\nabla_{\zeta(\theta)}(\epsilon v(\theta)).
$$

Sum them, we obtain 
$$
D_y\Upsilon(v, \xi)(\theta)=[\widetilde{\frac{\del\Phi}{\del \epsilon}}, X_{\eta(\theta)}]|_{\epsilon=0}+X_{\xi(\theta)}(z(\theta)).
$$

In particular, when $\eta\equiv\eta_0$ is constant, the Lie bracket term $[\widetilde{\frac{\del\Phi}{\del \epsilon}}, X_{\eta(\theta)}]|_{\epsilon=0}$ can be expressed as the Lie derivative 
$$
[\widetilde{\frac{\del\Phi}{\del \epsilon}}, X_{\eta(\theta)}]|_{\epsilon=0}=(\CL_{-X_{\eta_0}}v)(z(\theta)),
$$
since 
$$
\varphi_t(z(\theta)):=z(\theta+t)
$$ is the flow with initial point $z(\theta)$ of the vector field $-X_{\eta_0}$. Further 
we can write 
$$
(\CL_{-X_{\eta_0}}v)(z(\theta))=\frac{d}{ds}|_{s=0}\varphi_{-s *}v(z(\theta+s)).
$$

\end{proof}

The following lemma is useful for later derivation of the structure of $\ker D_y\widetilde\Upsilon$.

\begin{lem}\label{lem:vertical}If the vector field $v$ along $z$ is generated by the Lie algebra $\CG$, i.e., 
$$v(\theta)=X_{\zeta(\theta)}(z(\theta))
$$ 
for some $\zeta:I\to \CG$ on some interval or $S^1$, then we have 
$$
(\CL_{-X_{\eta_0}}v)(z(\theta))=X_{[\eta_0, \zeta(\theta)]}(z(\theta))+X_{\frac{d\zeta}{d\theta}}(z(\theta)).
$$
\end{lem}
\begin{proof}Continue to the notation and the proof in the above Lemma \ref{lem:linearization}. 
Notice that for our current case, the flow of the vector field $-X_{\eta_0}$ with initial point  $z\in X$ can be explicitly expressed as 
$$
\varphi_t(z)=\exp(-t\eta_0)z, \quad t\in \R,
$$ 
hence we are able to calculate out $(\CL_{-X_{\eta_0}}v)(z(\theta))$ explicitly as follows. 

First, we have 
\beastar
\varphi_{-s *}v(z(\theta+s))&=&\varphi_{-s *}[X_{\zeta(\theta+s)}(z(\theta+s))]\\
&=&\frac{d}{d\epsilon}|_{\epsilon=0}[\exp(s\eta_0)\exp(\epsilon\zeta(\theta+s))z(\theta+s)]\\
&=&X_{\exp(s\eta_0)\zeta(\theta+s)\exp(-s\eta_0)}(z(\theta)).
\eeastar
Then it follows
\beastar
(\CL_{-X_{\eta_0}}v)(z(\theta))&=&\frac{d}{ds}|_{s=0}\varphi_{-s *}v(z(\theta+s))\\
&=&\frac{d}{ds}|_{s=0}X_{\exp(s\eta_0)\zeta(\theta+s)\exp(-s\eta_0)}(z(\theta))\\
&=&X_{\frac{d}{d\theta}\zeta}(z(\theta))+X_{[\eta_0, \zeta(\theta)]}(z(\theta)).
\eeastar
\end{proof}

The following proposition is a key observation in \cite{L2} and we restate it here. 
\begin{prop}\label{prop:hess}Assume $y=(z, \eta)\in \Crit$ is a constant connection loop with $\eta\equiv \eta_0$.  
\begin{enumerate}
\item The space 
$$
\CX(y):=\{\CX_\tau=(-X_{\tau}, \dot{\tau}+[\eta, \tau])|\tau\in \Lie(LG)=L\CG\}
$$
 (i.e., the space generated by infinitesimal $\Lie(LG)$-actions, see \eqref{eq:infgauge}) lives in $\ker D_y\widetilde{\Upsilon}$. Decompose  
$$
\ker D_y\widetilde{\Upsilon}=\CX(y)\oplus (\ker D_y\widetilde\Upsilon\cap \CV_y).
$$
Then $\ker D_y\widetilde\Upsilon\cap \CV_y$ is isomorphic to the tangent space of (un)twisted sectors $M^{(\Hol(\eta_0))}$ of the orbifold $M$ at $[z(0)]$. The isomorphism is explicitly constructed as in the following proof. 
\item The so-defined formal Hessian operator $\Hess_y(v, \xi):=(J\Upsilon(v, \xi), d_z\mu(v))$,
$$\Hess_y: L^2(S^1, z^*TX\oplus  \CG)\to L^2(S^1, z^*TX\oplus  \CG)$$ 
with domain $W^{1, 2}(S^1, z^*TX\oplus  \CG)\subset L^2(S^1, z^*TX\oplus  \CG)$ is an  unbounded essentially self-adjoint operator.
Moreover, the restriction of $\Hess_y$ on the $L^2$-completion of $\CV_y$ is also  essentially self-adjoint.
\item The spectrum of the restriction of $\Hess_y$ on $L^2(\CV_y)$ contains only simple real eigenvalues which belong to $\{\frac{k}{|\Hol_{\eta}|}|k\in \Z\}$.
\end{enumerate}
\end{prop}
\begin{proof}
\begin{enumerate}
\item The fact $\CX(y)\subset \ker D_y\Upsilon$ immediately follows from Lemma \ref{lem:vertical}. 

The elements $(v, \xi)\in \ker D_y\widetilde\Upsilon\cap \CV_y$ satisfy the following three equations:
\begin{enumerate}
\item $\CL_{-X_{\eta_0}}v+X_{\xi(\theta)}=0$;
\item $d_z\mu(v)=0$;
\item $d\mu\circ J(v)+\dot{\xi}+[\eta_0, \xi]=0$.
\end{enumerate}
Since $0$ is a regular value of $\mu$, the equation (b) indicates $v\in T\mu^{-1}(0)$. To make the geometry more clear, we split the proof by considering the trivial holonomy case and the nontrivial holonomy case. 
\begin{description}
\item[The trivial holonomy case] 
For this case, $\eta_0=0$, and
the equation (a) and (c) become
$$
\dot{v}+X_{\xi(\theta)}=0\quad \text{and}\quad d_z\mu(v)=0.
$$
The infinitesimal action is free on $\mu^{-1}(0)$ together with $v$ being defined on $S^1$ forces $\xi\equiv 0$, and the second equation shows that
$v\in \ker d_z\mu\circ J=H_y$.

By this way, such pair $(v, \xi)$ corresponds to a vector living in the untwisted sector of the orbifold $M$. 
\item[The nontrivial holonomy case]
Denote by $\varphi_t(z)=\exp(-t\eta_0)z$ the flow with initial point $z$ of the vector field $-X_{\eta_0}$ as before. 

Define 
$$
\tilde{v}: S^1\to T_{z(0)}\mu^{-1}(0), \quad \tilde{v}(t)=\varphi_{-t, *}(v(t)).
$$
Here it lives in the subspace $T_{z(0)}\mu^{-1}(0)$ due to $\mu^{-1}(0)$ is $G$-invariant. The tangent space $T_{z(0)}\mu^{-1}(0)$ splits as 
$$
T_{z(0)}\mu^{-1}(0)=X_{\CG}(z(0))\oplus H(z(0)), 
$$
where $X_{\CG}$ denotes the subspace generated by infinitesimal $\CG$-action and $H=\ker d\mu\circ J$ is defined as before. We write $\tilde{v}=\tilde{v}_G+\tilde{v}_H$ with respect to the decomposition.
Due to the $G$-equivariance of $\mu$ and the $G$-invariance of $J$, 
\beastar
v_H(\theta)&:=&\varphi_{\theta, *}\tilde{v}_H(\theta)\in H(z(\theta)), \\
v_G(\theta)&:=&\varphi_{\theta, *}\tilde{v}_G(\theta)\in X_{\CG}(z(\theta)), \quad \theta\in S^1.
\eeastar
As a result, equation (a) indicates $\frac{d}{d\theta}\tilde{v}_H=0$ and so 
$\tilde{v}_H\equiv \tilde{v}(0)$,
which exactly means 
$$
v_H(0)=d\Hol_{\eta_0}(v_H(0))$$ 
by the expression of $\varphi_t$, hence $v_H(0)$ lives in the tangent space of the fixed point set 
${\mu^{-1}(0)}^{\Hol_{\eta_0}}$ by $\Hol_{\eta_0}$.

Next, we show $v_G\equiv 0$ and $\xi\equiv 0$. For this, write $v_G(\theta)=X_{\zeta(\theta)}$ for some $\zeta:S^1\to \CG$ (by the assumption that $0$ is a regular).
Then apply Lemma \ref{lem:vertical}, the equation (a) can be written as 
$$
X_{\frac{d}{d\theta}\zeta(\theta)}+X_{[\eta_0, \zeta(\theta)]}+X_{\xi(\theta)}=0,
$$
which is equivalent to 
\be\label{eq:Geq}
{\frac{d}{d\theta}\zeta(\theta)}+[\eta_0, \zeta(\theta)]+\xi(\theta)=0.
\ee
Consider the $L^2(S^1)$ inner product, 
\beastar
0&=&\int_{S^1}\langle \xi(\theta), {\frac{d}{d\theta}\zeta(\theta)}+[\eta_0, \zeta(\theta)]+\xi(\theta)\rangle d\theta\\
&=&-\int_{S^1}\langle \zeta(\theta), {\frac{d}{d\theta}\xi(\theta)}+[\eta_0, \xi(\theta)]\rangle d\theta+\|\xi\|^2_{L^2(S^1)}
\eeastar
Then by the slice equation (c), this becomes
\beastar
0=\int_{S^1}\langle \zeta(\theta),d\mu\circ J(X_{\zeta(\theta)})\rangle d\theta+\|\xi\|^2_{L^2(S^1)}
=\|X_{\zeta(\theta)}\|^2_{L^2(S^1, X)}+\|\xi\|^2_{L^2(S^1)},
\eeastar 
which shows the vanishing of both $\zeta$ and $\xi$.
\end{description}

Clearly from the expression of $D_y\widetilde\Upsilon$, $\ker D_y\widetilde\Upsilon\cap \CV_y$ is $\fg_{S^1}(y)$-invariant. Further recall from Lemma \ref{lem:isotropy} (2) that $\fg_{S^1}(y)\cong C_{G_{z(0)}}(\Hol_\eta)$, these finish the proof of the first statement.

\item Take any $(v_i, \xi_i)\in W^{1, 2}(S^1, z^*TX\oplus \CG)$, $i=1, 2$, we calculate
\beastar
&&\langle\Hess_y(v_1, \xi_1), (v_2, \xi_2)\rangle\\
&=&\int_{S^1}\langle J(\CL_{-X_{\eta_0}}v_1+X_{\xi_1}), v_2\rangle +\int_{S^1}\langle d_z\mu(v_1), \xi_2\rangle\\
&=&-\int_{S^1}\langle \CL_{-X_{\eta_0}}v_1, Jv_2\rangle -\int_{S^1}\langle X_{\xi_1}, Jv_2\rangle+\int_{S^1}\langle d_z\mu(v_1), \xi_2\rangle\\
&=&-\int_{S^1}d\langle v_1, Jv_2\rangle+\int_{S^1}\langle v_1, J\CL_{-X_{\eta_0}}v_2\rangle-\int_{S^1}\langle X_{\xi_1}, Jv_2\rangle+\int_{S^1}\langle d_z\mu(v_1), \xi_2\rangle\\
&=&0+\int_{S^1}\langle v_1, J\CL_{-X_{\eta_0}}v_2\rangle-\int_{S^1}\langle \xi_1, d_z\mu(v_2)\rangle+\int_{S^1}\langle \xi_2, d_z\mu(v_2)\rangle\\
&=&\langle\Hess_y(v_1, \xi_1), (v_2, \xi_2)\rangle.
\eeastar
Here we use the metric and $J$ are both $G$-invariant, and $z$ is the flow of $-X_{\eta_0}$ which is due to $(z, \eta_0)\in \Crit$. 

This shows $\Hess_y$ is  essentially self-adjoint  as  $\Hess_y$  is symmetric with a dense domain $W^{1, 2}(S^1, z^*TX\oplus \CG)$ and $\Hess_y: W^{1, 2}(S^1, z^*TX\oplus \CG) \to \to L^2(S^1, z^*TX\oplus  \CG)$ is self-adjoint. 

Since $\CX_y\subset \ker \Hess_y$ from (1) and $\Hess_y$ is self-adjoint, it follows the $L^2$-completion $\CV_y$  is $\Hess_y$-invariant, and $\Hess_y$ is also essentially  self-adjoint on  the $L^2$-completion $\CV_y$. 

\item Denote by $$
v=(v_H, X_\zeta(z), JX_{\alpha}(z)):S^1\to z^*H\oplus \Img(L_z)\oplus \Img(JL_z), \quad \zeta, \alpha: S^1\to \CG,
$$
and notice that $L^2(S^1, z^*H)$ and $L^2(S^1, \Img(L_z)\oplus \Img(JL_z))\oplus L\CG$ are two $\Hess_y$-invariant subspaces. 
Over $L^2(S^1, z^*H)$, 
$$
\Hess_y(v_H)=J\frac{d}{d\theta}\tilde{v}_H,
$$
whose spectrum contains only simple real eigenvalues as $\{\frac{k}{|\Hol_{\eta}|}|k\in \Z\}$ since its $|\Hol_{\eta}|$-multiple covering is conjugate to the self-adjoint operator 
$$
\frac{\sqrt{-1}}{|\Hol_{\eta}|}\frac{d}{d\theta}: L^2(S^1, \C^{\dim_{\C}H})\to L^2(S^1, \C^{\dim_{\C}H}).
$$

To understand the part over $L^2(S^1, \Img(L_z)\oplus \Img(JL_z))\oplus L\CG$, it turns out to be easier to work on the normal form introduced in Section \ref{sec:darboux}. Write the $\Hess_y$ into the following expression
$$
\Hess_y(X_\zeta(z), JX_{\alpha}(z), \xi)=(-X_{{\frac{d}{d\theta}\alpha}+[\eta_0, \alpha(\theta)]+\xi(\theta)}, JX_{{\frac{d}{d\theta}\zeta}+[\eta_0, \zeta(\theta)]+\xi(\theta)} , d\mu\circ J(X_{\alpha})(z)), 
$$
and further write $d\mu\circ J(X_\zeta)=\zeta$ via the normal form, after using the flow of $-X_{\eta}$ to pull-back everything, we obtain the following simpler expression as
$$
\widetilde{Hess}_y(X_{\tilde\zeta}(z), JX_{\tilde\alpha}(z), \tilde\xi)=(-X_{{\frac{d}{d\theta}\tilde\alpha}+\tilde\xi(\theta)}, JX_{{\frac{d}{d\theta}\tilde\zeta}+\tilde\xi(\theta)} , \tilde\alpha)).
$$ 
As a result, the Hessian operator can be identified with 
$$
(\tilde\zeta, \tilde\alpha, \tilde\xi)\mapsto (-\frac{d}{d\theta}\tilde\alpha-\tilde\xi, \frac{d}{d\theta}\tilde\zeta+\tilde\xi, \tilde\alpha),
$$
which is a first order elliptic  differential operator plus a compact perturbation when restricted to the slice $\CV_y$. It follows that 
the spectrum of the restricted $\Hess_y$ contains only simple real eigenvalues which belong to $\{\frac{k}{|\Hol_{\eta}|}|k\in \Z\}$.

\end{enumerate}

\end{proof}

\begin{rem}
In \cite{L2}, the authors define a $S^1$-valued functional whose critical points corresponds to the critical loops here. Moreover, this functional is infinitesimal $L\CG$-invariant. From this point of view, the property that the kernel of $D_y\widetilde{\Upsilon}$ includes the vector fields generated by the infinitesimal gauge action is expected. 

In this paper, we don't use this functional out of some technical reason (see Remark \ref{rem:techreason}), but the functional we defined in Section \ref{sec:functional} is heurisitically the same as the functional in \cite{L2} (see also Remark \ref{rem:isop}).   
\end{rem}

\section{The local functional on loops in $X\times \CG$ near critical loops}\label{sec:functional}
\subsection{The canonical gauge transformations along a path of loops}
We fix a path of loops 
$$
w=(u, \eta): [t_0, \infty)\to L(X\times \CG)
$$
over some interval $[t_0, \infty)\subset \R$ with $\|\mu\circ u\|_{C^0([t_0, \infty)\times S^1)}$ sufficiently small, so that 
we are able to work under the normal form $(N^\epsilon, \omega_0, G, J, \mu_0)$ as introduced in Section \ref{sec:darboux}.
Then for each $t\in [t_0, \infty)$, we can write 
$$
w(t)=(u(t), \eta(t))=((z(t), JL_{z(t)}\xi(t)), \eta(t)),
$$
with $z(t)$ as a loop in $\mu^{-1}(0)$  and $\xi(t)$ as a loop in $\CG$ for each $t\in [t_0, \infty)$. 
In particular, we use $y(t)$ to denote the pair 
$$
y(t):=(z(t), \eta(t)): S^1\to \mu^{-1}(0)\times \CG, \quad t\in [t_0, \infty).
$$

Using the notations in Section \ref{sec:critical} around \eqref{eq:h}, and denote by $\tilde{w}=(\tilde{u}, \tilde{\eta})$ the $[t_0, \infty)$-path of loops after the $[t_0, \infty)$-family of $L_0G$-gauge transformations as
\beastar
\tilde{u}(t)&:=&h_{\eta(t)}\cdot u(t)=(h_{\eta(t)}^{-1}z(t), JL_{h_{\eta(t)^{-1}}z(t)}(\Ad_{h_{\eta(t)}^{-1}}\xi(t))),\\
\tilde{\eta}(t)&:=&h_{\eta(t)}\cdot \eta(t)\equiv-\Log\Hol_{\eta(t)}.
\eeastar
In particular, we notice that $\tilde{u}(t, 0)=u(t, 0)$. 

The following Lemma is obvious due to Lemma \ref{lem:da} and the metric is $G$-invariant.
\begin{lem}For every $t\in [t_0, \infty)$, 
$$
\|d_{\tilde{\eta}(t)}\tilde{u}(t)\|_{C^\infty(S^1)}=\|d_{\eta(t)}u(t)\|_{C^\infty(S^1)}.
$$
\end{lem}

Also, we have 
\begin{lem}\label{lem:holonomy}
$$
\dist(u(t, 0), \Hol_{\eta(t)}^{-1}u(t, 0))\leq \|d_{\eta(t)}u(t)\|_{C^\infty(S^1)}.
$$
\end{lem}
\begin{proof}
\beastar
\dist(u(t, 0), \Hol_{\eta(t)}^{-1}u(t, 0))&\leq& \int_0^1 |\frac{\del}{\del s}(\Psi^{-1}_{\eta(t)}(s)u(t, s))|ds\\
&=& \int_0^1 |\ell_{\Psi_{\eta(t)}(s)*}(\frac{\del u}{\del s}(t, s)+X_{\eta(t)}(u(t, s)))|ds \\
&=&\int_0^1 |\frac{\del u}{\del s}(t, s)+X_{\eta(t)}(u(t, s))|ds\\
&\leq&\|d_{\eta(t)}u(t)\|_{C^\infty(S^1)}.
\eeastar
\end{proof}

Then we obtain the following useful lemma which states that we can smoothly and uniformly  perturb $(\tilde{u}(t), \tilde{\eta}(t))$ to obtain a smooth path of critical loops nearby. 

\begin{lem}\label{lem:nearbycrit}There exists some $\epsilon>0$, such that whenever $\|d_{\eta(t)}u(t)\|_{C^\infty(S^1)}\leq \epsilon$ for any $t\in [t_0, \infty)$ and $|\mu(u(t))|=|\xi(t)| \leq \epsilon$,  then there exists a unique smooth path $\eta_0:[t_0, \infty)\to \CG$ with each $\exp(\eta_0(t))\in G_{u(t, 0)}$ satisfying
\begin{enumerate}
\item $\|\tilde{\eta}(t)-\eta_0(t)\|_{L^\infty([t_0, \infty)}$ is small and controlled by $\epsilon$;
\item the loop $\exp(\theta\eta_0(t))\tilde{u}(t, \theta)$ stays in the injective radius of the initial point $\tilde{u}(t, 0)=u(t, 0)$. 
\end{enumerate}
\end{lem}

\begin{proof}We first construct $\eta_0$ for each $t\in [t_0, \infty)$. As seen from Lemma \ref{lem:holonomy}, 
$$
\dist_{G\cdot u(t, 0)}(u(t, 0), \Hol_{\eta(t)}^{-1}u(t, 0))\leq \|d_{\eta(t)}u(t)\|_{C^\infty(S^1)}\leq \epsilon,
$$ 
where $G\cdot u(t, 0)$ denotes the $G$-orbit of $u(t, 0)$ in $\mu^{-1}(0)$. WLOG, we can assume  
$G\cdot u(t, 0)$ is isometric to $G/G_{u(t, 0)}$, and then it follows 
$$
\dist_{G/G_{u(t,0)}}([e], [\Hol_{\eta(t)}])\leq \epsilon. 
$$  
Hence there exists some $k_t\in G_{u(t,0)}$ with $\dist_{G}(k_t, \Hol_{\eta(t)})\leq \epsilon$. 
Moreover, whenever we make $\epsilon<d_{\mu^{-1}(0), G}$ where 
$$
d_{\mu^{-1}(0), G}=\inf_{x\in \mu^{-1}(0)}d_{x, G} \quad \text{ with } \quad d_{x, G}:=\inf_{k, k'\in G_x}\dist_G(k, k')
$$ as introduced in \eqref{eq:Gdist}, such $k_t\in G_{u(t, 0)}$ is unique. 

Now as come to the $[t_0, \infty)$-family of $\eta(t)$'s,  since $\eta(t)$ is smooth in $t$, by the construction, $k_t$ is also smooth in $t$. Because $G$ is a  connected compact Lie group, we can further take a smooth family $\eta_0(t)$'s such that 
$$
\exp(\eta_0(t))=k_t, \quad |\tilde{\eta}(t)-\eta_0(t)|\leq C\dist_G(\Hol_{\eta(t)}, k_t)\leq C\epsilon  
$$ 
for every $t\in [t_0, \infty)$, where $C$ is a constant only depends on the geometry $(X, \omega, G, J, \mu, M)$. 
Hence we are done with Property (1). 

From the construction,  we have 
$$
\exp(\theta\eta_0(t))\tilde{u}(t, \theta)=\exp(\theta(\tilde{\eta}(t)-\eta_0(t)))\Psi_{\eta(t)}^{-1}(\theta)u(t, \theta).
$$ 
Then we estimate
\beastar
\dist(\exp(\theta\eta_0(t))\tilde{u}(t, \theta), z(t, 0))
&=&\int_{0}^{\theta}|\frac{\del}{\del s}\exp(s\eta_0(t))\tilde{u}(t, s)|ds\\
&=&\int_{0}^{\theta}|\frac{\del}{\del s}(\exp(s(\tilde{\eta}(t)-\eta_0(t)))\Psi_{\eta(t)}^{-1}(s)u(t, s))|ds\\
&\leq&C'|\tilde{\eta}(t)-\eta_0(t)||d_{\eta(t)}u(t, s)|
\eeastar
with $C'$ some constant determined by the geometry data $(X, \omega, G, J, \mu, M)$. Hence the distance is bounded by $\epsilon^2$ and we can make $\epsilon$ small so that $\tilde{u}(t)$ lives in the injective radius of $u(t, 0)$ for every $t\in [t_0, \infty)$.  
\end{proof}

Denote the open or closed disks of radius $r$ by
$$
D(r)=\{z\in \C||z|< r\}, \quad \overline{D(r)}=\{z\in \C||z|\leq r\},
$$ 
and the boundary by $\del D(r)=\overline{D(r)}\setminus D(r)=\{z||z|=r\}$. Using above lemmas, we now construct a function (in fact a local functional) as stated in the following lemma. 

\begin{lem}\label{lem:offshell}Assume $\|d_{\eta(t)}u(t)\|_{C^\infty(S^1)}\leq \epsilon$ and 
$\|\mu(u(t))\|_{C^\infty(S^1)}\leq \epsilon$ for any $t\in [t_0, \infty)$  as the $\epsilon$ given in Lemma \ref{lem:nearbycrit}. For each $t\in [t_0, \infty)$, there exists a piecewise smooth map $\uu_t: D(2)\to X$ with boundary $u(t)=\uu_t|_{\del D(2)}$ 
satisfying
\begin{enumerate}
\item \begin{enumerate}
\item $\uu_t|_{\overline{D(1)}}\subset \mu^{-1}(0)$ and smooth, with $\uu_t|_{\del D(1)}=z(t)$;
\item $\uu_{t}|_{\overline{D(2)}-D(1)}(s, \theta)=(z(t, \theta), (s-1)JL_{z(t, \theta)}(\xi(t, \theta)))$, 
where  $(s, \theta)$ denotes the polar coordinate for the disk $D(2)$. 
\end{enumerate}

\item Denote by $$
g_t: S^1\to G, \quad g_t(\theta):=h_{\eta(t)}(\theta)\exp(-\theta\eta_0(t))
$$
a $[t_0, \infty)$-family of $[0, 2\pi]$-gauge transformations. Then 
$g_t^{-1}z(t, \cdot)$ lives in the injective radius of $z(t, 0)$. 
\end{enumerate}
Using such $\uu_t$, we define the function $\CL_w:[t_0, \infty)\to \R$ as  
\be\label{eq:L}
\CL_w(t):=-\int_{D(2)}\uu_t^*\omega+\int_{S^1}<\mu(u(t, \theta)), \eta(t, \theta)>d\theta.
\ee
Then $\CL_w$ is smooth in $t$  and 
satisfies the following isoperimetric inequality:
\be\label{eq:isoperimetric}
\CL_w(t)\leq c_0(\|d_{\eta(t)}u(t)\|^2_{p}+c_1\|\mu\circ u(t)\|^2_{\frac{p}{p-1}}),
\ee
where $c_0, c_1>0$ are constants only depending on $(X, \omega, G, \mu, J, M)$ and $1\leq p\leq 2$. 

\end{lem}

\begin{proof}From Lemma \ref{lem:nearbycrit}, the loop $g_t^{-1}z(t, \cdot)$ for every $t\in [t_0, \infty)$ lives in the injective radius of $z(t, 0)$ and hence contractible. As a result, we can take a smooth bounding disk within the injective radius of $z(t, 0)$ as
$$
\bar{u}^1_t: D(1)\to \mu^{-1}(0)
$$
whose boundary is this loop $g_t^{-1}z(t, \cdot)$. Then define 
$$
\bar{u}^2_t|_{D(2)-D(1)}(s, \theta)=(g_t^{-1}z(t, \cdot), (s-1)JL_{g_t^{-1}z(t, \cdot)}(\Ad_{g_t^{-1}}\xi(t, \theta))).
$$
Denote by $\bar{u}_t$ the piecewise smooth map connecting $\bar{u}^1_t$ and $\bar{u}^2_t$. 
Then 
$\uu_t:=g_t^{-1}\bar{u}_t: D(2)\to X$ is the piecewise smooth map satisfying the two properties as stated in (1), for every $t\in [t_0, \infty)$. 

Next we prove the function $\CL_w: [t_0, \infty)\to \R$ defined as \eqref{eq:L} satisfies the isoperimetric type inequality \eqref{eq:isoperimetric}.

We first prove the following lemma. 
\begin{lem}\label{lem:gaugeindep}
\bea
\CL_w(t)
&=&-\int_{D(2)}\bar{u}_t^*\omega+\int_{S^1}<\mu(\bar{u}(t, \theta)), (g_t^{-1}\cdot\eta)(t, \theta)>d\theta\label{eq:small}
\eea
\end{lem}
\begin{proof}We calculate for every fixed $t\in [t_0, \infty)$, 
\beastar
&&\int_{D(2)}\bar{u}_t^*\omega-\int_{D(2)}\uu_t^*\omega\\
&=&\int_{D(2)}(g_t\uu_t)^*\omega-\uu_t^*\omega\\
&=&-\int_{D(2)}d\langle\mu\circ \uu_t, g_t^{-1}(dg_t) \rangle\\
&=&-\int_{\del D(2)}\langle\mu\circ u(t, \cdot), g_t^{-1}(dg_t) \rangle.
\eeastar
On the other hand, 
\beastar
&&\int_{S^1}<\mu(\bar{u}(t, \theta)), (g_t^{-1}\cdot\eta)(t, \theta)>d\theta-\int_{S^1}<\mu(u(t, \theta)), \eta(t, \theta)>d\theta\\
&=&\int_{S^1}<\mu(g_t\uu_t(\theta)), (g_t^{-1}\cdot\eta)(t, \theta)>d\theta-\int_{S^1}<\mu(u(t, \theta)), \eta(t, \theta)>d\theta\\
&=&\int_{S^1}<\Ad_{g_t^{-1}}\mu(\uu_t), \Ad_{g_t^{-1}}\eta-(dg_t)g_t^{-1}>-\int_{S^1}<\mu(u(t, \theta)), \eta(t, \theta)>d\theta\\
&=&\int_{S^1}(<\mu(u(t, \theta)), \eta(t, \theta)>-<\mu(u(t, \theta)), g_t^{-1}(dg_t)>)d\theta-\int_{S^1}<\mu(u(t, \theta)), \eta(t, \theta)>d\theta\\
&=&-\int_{S^1}<\mu(u(t, \theta)), g_t^{-1}(dg_t)>)d\theta.
\eeastar
We are done with this lemma. 
\end{proof}
Since the right hand side of \eqref{eq:small} is independent of choices of $\bar{u}_t^1$ but only depends on bounding loop $\tilde{z}_t$ which smoothly depends on $t$, it follows $\CL_w$ is smooth in $t$. 
On the other hand, using Lemma \ref{lem:gaugeindep} and  the metric is $G$-invariant, we only need to prove \eqref{eq:isoperimetric} for $\bar{u}_t$ since $d_{g\cdot\eta}(g^{-1}u)=\ell_{g^{-1*}}(d_\eta u)$ and $\mu$ is $G$-equivariant. 

First from the isoperimetric inequality for small symplectic disks living in the injective radius of $\bar{u}(t, 0)$, we have 
the following standard symplectic isoperimetric inequality for any $1\leq p\leq 2$, 
\be\label{eq:sympisop}
-\int_{D(2)}\bar{u}_t^*\omega\leq c\|d\bar{u}_t|_{\del D(2)}(\theta)\|_{p}^2.
\ee
 The other term, 
 $$
 \int_{S^1}<\mu(\bar{u}(t, \theta)), (g_t\cdot\eta)(t, \theta)>d\theta\leq c'\|\mu(\bar{u}_t)\|_{\frac{p}{p-1}}^2+c'\|g_t\cdot\eta\|_{p}^2
 $$
 On the other hand, since the infinitesimal action of $\CG$ on $X$ is free near $\mu^{-1}(0)$ when $0$ is regular, we have 
 $$
 \|g_t\cdot\eta\|_{p}^2\leq \|X_{g_t\cdot\eta}(\bar{u}_t)\|_{p}^2.
 $$
 Then by taking $c_0, c_1$ large enough, we have 
 $$
 -\int_{D(2)}\bar{u}_t^*\omega+ \int_{S^1}<\mu(\bar{u}(t, \theta)), (g_t\cdot\eta)(t, \theta)>d\theta\leq c_0(\|d_{g_t\cdot\eta}\bar{u}_t|_{\del D(2)}(\theta)\|_{p}^2+c_1\|\mu(\bar{u}_t)\|_{\frac{p}{p-1}}^2).
 $$
The whole proof is done now.
\end{proof}

\begin{rem}\label{rem:isop}This isoperimetric inequality \eqref{eq:isoperimetric} can be made with sharp constants as in \cite{ziltener-exp}.  For example, the constant $c$ in \eqref{eq:sympisop} holds for any $c>\frac{1}{4\pi}$. 
Also, following the scheme as in \cite{ziltener-exp}, one is also able to prove this is a well-defined local functional. In fact, this functional is the same as the one introduced in \cite{L2} (see also \cite{frauenfelder1, frauenfelder2}) by using e.g., $z(t_0, \cdot): S^1\to \mu^{-1}(0)$ with capping $\uu_{t_0}|_{\overline{D(1)}}$, as the reference loop, since changing the reference loops in the same homotopy class doesn't change the local functional. 
We leave them to interested readers since these are not needed for our current application. 
\end{rem}

\subsection{An energy equality for symplectic vortices}

Using Lemma \ref{lem:offshell}, for each pair $w=(u, \eta): [t_0, \infty)\times S^1\to X\times \CG$ satisfying 
\be\label{eq:assump}
\|d_{\eta(t)}u(t)\|_{C^\infty(S^1)}\leq \epsilon, \quad \|\mu(u(t))\|_{C^\infty(S^1)}\leq \epsilon
\ee 
for any $t\in [t_0, \infty)$  as the $\epsilon$ given in Lemma \ref{lem:nearbycrit}, we have constructed a function 
$$
\CL_w:[t_0, \infty)\to [0, \infty)
$$
as defined in Lemma \ref{lem:offshell}, which satisfies the isoperimetric inequality \eqref{eq:isoperimetric}. In particular, we remark that the constants in  inequality \eqref{eq:isoperimetric} are independent of $w$ but only depend on the geometry $(X, \omega, G, J, \mu, M)$. 

Now we further assume $w=(u, \eta)$ with 
$$
u: [t_0, \infty)\times S^1\to X, \quad \eta: [t_0, \infty)\times S^1\to\CG
$$ 
is a (temporal) symplectic vortex over the half cylinder end $[t_0, \infty)\times S^1$.
Then the function $\CL_w$ is related to the YMH energy as stated in the following proposition. 
\begin{prop}\label{prop:onshell}Suppose $w=(u, \eta)$ is a (temporal) symplectic vortex satisfying the assumptions \eqref{eq:assump}. 
Then  for any $[t_1, t_2]\subset[t_0, \infty)$, 
\be\label{eq:onshell}
E(w; [t_1, t_2]\times S^1):=\int_{[t_1, t_2]\times S^1}e(u, \eta)\nu_h=\CL_w(t_1)-\CL_w(t_2),
\ee
where $e(u, \eta)$ is the energy density defined as in \eqref{eq:e}.
\end{prop}
\begin{proof}
Denote by $\tilde{\uu}$ the path joining loops $u_{t_1}$ and $u_{t_2}$ by connecting the three pieces 
$\overline{\uu_{t_1}^2}, z_t$, and $\uu_{t_2}^2$, 
where $\uu_t^2$ is as defined in the proof of Lemma \ref{lem:offshell} and the `overline' of $\uu_t^2$ denotes the path by reversing orientation.  
By construction, $u$ and $\tilde{\uu}$ are homotopic relative to boundaries $u(t_1, \cdot)$ and $u(t_2, \cdot)$, and then it follows from \eqref{eq:energyeq} that 
$$
E(w; [t_1, t_2]\times S^1)=\int_{[t_1, t_2]\times S^1}\tilde{\uu}^*\omega-d\langle\mu(\tilde{\uu}), \eta\rangle.
$$

On the other hand, by considering the three parts $\overline{\uu_{t_1}^2}, z_t$, and $\uu_{t_2}^2$ that $\tilde{\uu}$ consists of, we have 
\beastar
\int_{[t_1, t_2]\times S^1}\tilde{\uu}^*\omega-d\langle\mu(\tilde{\uu}), \eta\rangle
&=&[\CL_w(t_1)-(\int_{D(1)}  (\tilde{\uu}_{t_1}^1)^*\omega-d\langle\mu(\tilde{\uu}_{t_1}^1), \eta \rangle)]\\
&+&(\int_{D(1)} z^*\omega-d\langle\mu(z), \eta\rangle)\\
&-&[\CL_w(t_2)-(\int_{D(1)}(\tilde{\uu}_{t_2}^1)^*\omega-d\langle\mu(\tilde{\uu}_{t_2}^1), \eta \rangle)]\\
&=&\CL_w(t_1)-\CL_w(t_2)+\int_{S^2}(\uu^1)^*\omega.
\eeastar
Here the notations $\tilde{\uu}_{t_1}^1$, $\tilde{\uu}_{t_2}^1$ are the same as in the proof of Lemma \ref{lem:offshell},
and  $\uu^1$ is a path of loops constructed by joining $\overline{\uu_{t_1}^1}$, $z$, $\uu_{t_2}^1$. 
In particular, $\uu^1$ can be transformed to a contractible sphere by a $[t_0, \infty)$-family of $[0, 2\pi]$-gauge transformations as 
in Lemma \ref{lem:nearbycrit} and Lemma \ref{lem:offshell}. Notice that $\uu^1$ stays in $\mu^{-1}(0)$, it follows the symplectic area from the last term is invariant under such family of gauge transformations and hence is zero. 
We are done with the proof then. 

\end{proof}

It follows from Proposition \ref{prop:onshell} that $\lim_{t\to \infty}\CL_w(t)=:\CL_w(\infty)$ exists when $w$ is a temporal symplectic vortex (notice that we haven't shown the $C^0$-convergence of $w$). Further, together with \eqref{eq:isoperimetric}, we obtain
\begin{cor}Under the same condition as in Proposition \ref{prop:onshell}, 
$\CL_w(\infty)=0$.
\end{cor}

We end this section by remarking that so far we haven't required any property on the metric over $[t_0, \infty)\times S^1$.

\section{The asymptotic convergence}\label{sec:thm}

In this section, we prove the main theorem \ref{thm:main}, which we split into the following two statements. 

\begin{prop}\label{prop:main}Assume $(P, w)=(P, u, A)$ is a admissible finite energy symplectic vortex over the punctured Riemann surface $(\dot\Sigma, j, h)$ with $k$-punctures. 
Then there exist some gauge transformation $\Phi\in \fg_P$,  $k$ critical loops $x_\infty^i=(z_\infty^i, \eta_\infty^i)$, $i=1, \cdots, k$, and some constants $C>0$, $\delta>0$ which only depend on $(X, \omega, G, \mu, J, M)$, such that 
$\Phi\cdot w=(\Phi^*u, \Phi^*A)$ is temporal and for any $\theta\in S^1$
\begin{enumerate}
\item $|d_{\Phi^*A} \Phi^*u(t, \theta)|\leq Ce^{-\delta t}$, 
\item $|\mu(\Phi^*u)(t, \theta)|\leq Ce^{-\delta t-bt}$,
\item $|F_{\Phi^*A}(t, \theta)|\leq Ce^{-\delta t+bt}$,
\item $\dist(\Phi\cdot u(t, \theta), z_\infty(\theta))\leq Ce^{-\delta t}$,
\item $\|\Phi^*A(t, \cdot)-\eta_\infty(\cdot)\|_{L^p(S^1)}\leq Ce^{-(\delta+b(\frac{2}{p}-1))t}$, for $p\geq 2$ and $\delta+b(\frac{2}{p}-1)>0$.
\end{enumerate}
Here we use $z_\infty, \eta_\infty, \delta, b$ to denote the ones with corresponding superscripts $i$'s, $i=1, \cdots, k$.
\end{prop} 

\begin{cor}\label{cor:main}Under the same assumption as in Proposition \ref{prop:main}, Then there exist some gauge transformation $\Phi\in \fg_P$,  $k$ critical loops $x_\infty^i=(z_\infty^i, \eta_\infty^i)$ of \textbf{constant connections}, $i=1, \cdots, k$,  and some constants $C>0$, $\delta_i>0$ which only depend on $(X, \omega, G, \mu, J, M)$, such that 
$\Phi\cdot w=(\Phi^*u, \Phi^*A)$ is temporal and for any $\theta\in S^1$
\begin{enumerate}
\item $|d_{\Phi^*A} \Phi^*u(t, \theta)|\leq Ce^{-\delta t}$,
\item $|\mu(\Phi^*u)(t, \theta)|\leq Ce^{-\delta t-bt}$,
\item $|F_{\Phi^*A}(t, \theta)|\leq Ce^{-\delta t+bt}$,
\item $\|\dist(\Phi\cdot u(t, \cdot), z_\infty(\cdot))\|_{L^p(S^1)}+\|\Phi^*A(t, \cdot)-\eta_\infty(\cdot)\|_{L^p(S^1)}\leq Ce^{-(\delta+b(\frac{2}{p}-1)})t$, for $p\geq 2$ and $\delta+b(\frac{2}{p}-1)>0$.
\end{enumerate}
Here we use $z_\infty, \eta_\infty, \delta, b$ to denote the ones with corresponding superscripts $i$'s, $i=1, \cdots, k$.

\end{cor}
The rest of this section is devoted to the proof of them.

\subsection{The exponential decay of the $b$-energy density}

First we notice that it is enough to prove Theorem \ref{thm:main} over one end, so WLOG, we assume $w=(u, \eta)$ is a finite energy temporal symplectic vortex  over half cylinder $[0, \infty)\times S^1$ with respect to the metric 
$$
h=e^{2b(t)}(dt^2+d\theta^2), \quad (t, \theta)\in [0, \infty)\times S^1.
$$ 
Define the $b$-energy density as 
\bea\label{eq:eb}
e_b(w)&:& [0, \infty)\times S^1\to \R\\
e_b(w)&=&\frac{1}{2}(|d_\eta u|^2+e^{-2b(t)}|F_\eta|^2+e^{2b(t)}|\mu\circ u|^2),
\eea
where $|\cdot|$ denotes the norm with respect to the standard cylinder metric $dt^2+d\theta^2$. 
Then it follows that the YMH energy of $w$ over $[0, \infty)\times S^1$ is  
$$
E(u, \eta)=\int_{[0, \infty)\times S^1}e_b(w)dtd\theta<\infty.
$$
The following lemma is proved by Ziltener as \cite[Lemma 3.3]{ziltener-exp} based on the a priori estimates given by Gaio-Salamon as in \cite[Section 9]{gaio-salamon}.  We remark that both  \cite[Lemma 3.3]{ziltener-exp} and \cite[Section 9]{gaio-salamon} assume that the group $G$-action is free on the level set $\mu^{-1}(0)$. However in fact it applies to the general case by removing the free action assumption without changing any word, since the only assumption needed in the proof is the properness of moment map $\mu$ and the infinitesimal Lie algebra $\CG$ action is free which is satisfied by assuming $0$ is a regular value. 
\begin{lem}\label{lem:apriori}There exists some constant $c>0$ which only depends on $(X, \omega, G, \mu)$, such that for any finite energy temporal symplectic vortex $w=(u, \eta)$ over half cylinder $[0, \infty)\times S^1$ with respect to the metric 
$$
h=e^{2bt}(dt^2+d\theta^2), \quad (t, \theta)\in [0, \infty)\times S^1,\quad b\geq 0,
$$ 
there exists a constant $t_0>1$ such that the $b$-energy density 
$$
e_b(w)(t, \theta)\leq cE(w; [t-1, t+1]\times S^1), \quad \text{ for any } t\geq t_0.
$$
\end{lem}
\begin{rem}The admissible condition that is required as in  \cite[Lemma 3.3]{ziltener-exp} is more general than our current admissible condition. We restrict to this relatively simpler condition, i.e., to require the linearity of the exponent function $b(t)=bt$ with $b\geq 0$ in the conformal metric,  is because such cases are already enough for the expected applications in constructing corresponding moduli spaces. 
\end{rem}

In particular, followed from this Lemma, by assuming the metric $h$ is admissible, i.e., $b\geq 0$, we can make $t_0$ large enough, such that the assumptions \eqref{eq:assump} are satisfied, and as a consequence, the function $\CL_w: [t_0, \infty)\to [0, \infty)$ can be defined as given in \eqref{eq:L}.

Applying Proposition \ref{prop:onshell} and \eqref{eq:isoperimetric} (with $p=2$), we estimate
\beastar
E(w; [t, \infty))&=&\CL_w(t)\\
&\leq&c_0(\|d_{\eta}u(t)\|^2_{L^2(S^1)}+c_1\|\mu\circ(u)(t)\|^2_{L^2(S^1)})\\
&\leq&\frac{1}{2\delta}(\frac{1}{2}\|d_{\eta}u(t)\|^2_{L^2(S^1)}+e^{2b(t)}\|\mu\circ(u)(t)\|^2_{L^2(S^1)})\\
&=&-\frac{1}{2\delta}\frac{d}{dt}\int^\infty_t\int_{S^1}e_{b}(w)d\theta dt\\
&=&-\frac{1}{2\delta}\frac{d}{dt}E(w; [t, \infty)),
\eeastar
where $\delta>0$ is some constant determined by $c_0, c_1$. 
We remark here, for the second inequality uses the assumption that the metric $h$ is admissible. 
As a result, 
\be\label{eq:Edecay}
E(w; [t, \infty))\leq E(w)e^{-2\delta t} \text{ for any } t\geq t_0,
\ee
where $E(w)$ is the total energy of $w$ over $[0, \infty)\times S^1$.  

Then apply Lemma \ref{lem:apriori} again, we obtain that there exists some constant $C>0$ which only depends on $(X, \omega, G, J, \mu, M)$, such that 
\bea
|d_\eta u(t, \theta)|&\leq& Ce^{-\delta t}\label{eq:udecay}\\
|\mu(u)(t, \theta)|&\leq& Ce^{-\delta t-bt}\label{eq:mudecay}\\
|F_\eta(t, \theta)|&\leq& Ce^{-\delta t+bt}.\label{eq:Fdecay}
\eea

\subsection{The exponential convergence of $w$ to a critical loop}\label{sec:c0decay}

To obtain the limiting critical loop, we look at the translated sequences as below. For any sequence $\{t_i\}$ with $t_i\to \infty$ as $i\to\infty$, $i\in \CI=\Z \text{ or } \R $, define 
$$
w_i(t, \theta)=(u_i(t, \theta), \eta_i(t, \theta)):=w(t+t_i, \theta), \quad t\in [-t_i, \infty).
$$
A moment of checking shows that $\{w_i\}$ satisfy the equations 
\beastar
\frac{\del u_i}{\del t}+J(u_i)(\frac{\del u_i}{\del\theta}+X_{\eta_i}(u_i))&=&0\\
\frac{\del \eta_i}{\del t}+e^{2bt+2bt_i}\mu(u_i)&=&0,
\eeastar
for $(t, \theta)\in [-t_i, \infty)\times S^1$. 
In another word, each $w_i$ is a temporal symplectic vortex over $([-t_i, \infty)\times S^1, j_0, e^{2b_i(t)}(dt^2+d\theta^2))$ with 
$b_i(t)=b(t+t_i)$. Moreover,  they have a uniform energy bound, as for each $i$ the energy 
$$
E_{b_i}(w_i|_{K\times S^1})\leq E_b(w)\leq \infty
$$
over any compact subset $K\subset [-t_i, \infty)$.  Here we use $E_{(\cdot)}$ to distinguish the energy with respect to different metrics as $e^{2(\cdot)}(dt^2+d\theta^2)$.

The following lemma follows from \cite[Lemma 9.3]{gaio-salamon} by interpreting in cylinder coordinates which is adapted to our current situation. We remark that from \cite[Lemma 9.3]{gaio-salamon}, one can also obtain higher order $L^p$ bound which we don't list here since they are not used in this paper. 
\begin{lem}\label{lem:Lp}
For any $C_0>0$, there exist some constants $\epsilon_0>0$ and $c>0$, such that 
any temporal $(u, \eta)$ with $0<\epsilon\leq\epsilon_0$ satisfying the following equations
\bea\label{eq:translationeq}
\frac{\del u}{\del t}+J(u)(\frac{\del u}{\del\theta}+X_{\eta}(u))&=&0\nonumber\\
\frac{\del \eta}{\del t}+e^{2bt}\epsilon^{-2}\mu(u)&=&0
\eea
over $K:=[-2, 2]\times S^1$ together with $$
\|d_\eta u\|_{L^\infty(K)}+\epsilon^{-1}\|\mu(u)\|_{L^\infty(K)}<C_0, 
$$
we have over $K_0:=[0, 1]\times S^1$, $2\leq p\leq \infty$,
\bea
\|\mu(u)\|_{L^p(K_0)}&\leq& c\epsilon^{\frac{2}{p}+1}\|\frac{\del u}{\del t}\|_{L^2(K)}+c\epsilon^{\frac{2}{p}}\|\mu(u)\|_{L^2(K)}\label{eq:apriori1}.
\eea
Here all norms are with respect to the standard cylindrical metric $dt^2+d\theta^2$. 

In particular, if $b=0$, one can make $\epsilon_0=1$. 
\end{lem}

\begin{proof}
We write down the equation in the $\C$-coordinates as $z=e^{t+i\theta}$ and plug
$\lambda=e^{bt-t}$ into the inequality as given in \cite[Lemma 9.3]{gaio-salamon}, then the desired inequalities follow.
In particular, the powers of $\epsilon$ stay the same as in \cite[Lemma 9.3]{gaio-salamon}, since the regions $K, K_0$ we are considering here are both bounded which makes the coordinate change only effect coefficient $c$. 
\end{proof}
Now apply \eqref{eq:apriori1} by taking $\epsilon=e^{-bt_i}$, it immediately follows that for each $w_i$, 
\be\label{eq:mu}
\|\mu(u_i)\|_{L^p(K_0)}\leq ce^{-(\delta+b(\frac{2}{p}+1))t_i}.
\ee

From now on, we assuming that $p\geq 2$ such that $\delta':=\delta+b(\frac{2}{p}-1)>0$ and estimate for $\eta_i$, $i=1, 2, \cdots$ as follows. For any $\theta\in S^1$, 
\beastar
|\eta_i(1, \theta)-\eta_i(0, \theta)|&\leq&\int_0^1|\frac{\del\eta_i}{\del t}(t, \theta)|dt\\
&\leq& c_b\int_0^1|e^{2bt_i}\mu(u_i)(t, \theta)|dt\\
&\leq&c_b e^{2bt_i}\|\mu(u_i)(\cdot, \theta)\|_{L^p([0, 1])},
\eeastar
where $c_b$ is some constant only depending on $b$. Here the second inequality is due to \eqref{eq:translationeq}, and the third one is from the H\"older's inequality.
 Notice that for sequence $t_i=i$, $i=1, 2, \cdots$, $\eta_i(1, \theta)=\eta_{i+1}(0, \theta)$. It follows from the above estimate that 
\beastar
\|\eta_{i+1}(0, \cdot)-\eta_i(0, \cdot)\|_{L^p(S^1)}&=&(\int_{S^1}|\eta_i(1, \theta)-\eta_i(0, \theta)|^pd\theta)^{\frac{1}{p}}\\
&\leq&c_be^{2bt_i}\|\mu(u_i)\|_{L^p(K_0)}\\
&\leq&c_be^{-\delta' i},
\eeastar 
where the last inequality is \eqref{eq:mu}. 
This leads to the convergence of $\eta_i(0, \theta)$ to some $\eta_\infty\in L^p(S^1, \CG)$ in $L^p(S^1)$-norm as $i\to \infty$ with decay rate $\delta'$. 
Further applying \eqref{eq:translationeq} and \eqref{eq:mu}, we obtain 
$$
\|\eta(t, \cdot)-\eta_\infty(\cdot)\|_{L^p(S^1)}\leq c_b e^{-\delta' t}. 
$$

Similarly, for any $\theta\in S^1$,
\beastar
\dist(u_{i+1}(0, \theta), u_i(0, \theta))&\leq&\int_0^1|\frac{\del u_i}{\del t}(t, \theta)|dt\\
&\leq&ce^{-\delta i},
\eeastar
where the last inequality follows from \eqref{eq:udecay} applying to a temporal vortex. 
This proves that there exists some $z_\infty\in C^0(S^1, X)$ such that $u_i(0, \theta)$ converges to $z_\infty(\theta)$ exponentially fast with decay rate $\delta$.  Further using \eqref{eq:udecay} again, this convergence can be enhanced to 
$$
\dist(u(t, \theta), z_\infty(\theta)\leq ce^{-\delta t}, \quad \text{for any } \theta\in S^1,
$$
and also we have $\mu(z_\infty)=0$ since \eqref{eq:mudecay}.

Now we show $(z_\infty, \eta_\infty)$ is a limiting loop. Notice that from the convergences we have shown above, the vector field 
$X_{\eta(t, \theta)}(u(t, \theta))$ converges to the vector field $X_{\eta_\infty(\theta)}(z_\infty(\infty))$ via parallel transports. 
At the same time, we have the exponential decay of $|\frac{\del u}{\del\theta}(t, \theta)+X_{\eta(t, \theta)}(u(t, \theta))|$ to $0$ with rate $\delta$, it follows 
that $z_\infty\in W^{1, p}(S^1)$ and hence $C^{0, \alpha}(S^1)$ by the Sobolev embedding theorem, and 
$$
\dot{z}_\infty+X_{\eta_\infty}(z_\infty)=0.
$$
This finishes the proof of Proposition \ref{prop:main}.

\bigskip

Now we use the canonical gauge transformation defined by \eqref{eq:h} and obtain the $W^{1, p}(S^1, G)$ based gauges $h_{\eta(t)}, h_{\eta_\infty}\in L_0G$. Clearly, by the exponential $L^p(S^1)$ convergence of $\eta(t, \cdot)$,  we have 
\beastar
\dist_{L^p(S^1)}(\Hol_{\eta(t)}, \Hol_{\eta_\infty})&\leq& ce^{-\delta' t}\\
\dist_{L^p(S^1)}(h_{\eta(t)}, h_{\eta_\infty})+\|\frac{\del h_{\eta(t)}}{\del\theta}-\frac{\del h_{\eta_\infty}}{\del\theta}\|_{L^p(S^1)}&\leq &ce^{-\delta' t},
\eeastar
where $c>0$ is some constant (which may be bigger than the constant $c$ before). 

Now define the gauge transformation for the trivial bundle over $[0, \infty)\times S^1$ as  
$\Phi(t, \theta)=h_{\eta_\infty}(\theta)$.
It follows 
\beastar
\|\Phi^*\eta(t, \cdot)-h_{\eta_\infty}^*\eta_\infty\|_{L^p(S^1)}&\leq&ce^{-\delta' t}\\
\|\dist(\Phi^*u(t, \cdot), h_{\eta_\infty}^*z_\infty)\|_{L^p(S^1)}&\leq&ce^{-\delta' t}.
\eeastar
In particular, here $h_{\eta_\infty}^*\eta_\infty=-\Log\Hol_{\eta_\infty}$ is constant, and $(h_{\eta_\infty}^*z_\infty, h_{\eta_\infty}^*\eta_\infty)$ is a critical loop with constant connection. 

This finishes the proof of Corollary \ref{cor:main}.

\bigskip
We end this section by the following two remarks for the proof.

\begin{rem}\label{rem:techreason}
\begin{enumerate}
\item The decay rate $\delta$ can be enhanced to the minimal nonzero eigenvalues of the self-adjoint Hessian operator restricted to slices as in Proposition \ref{prop:hess} (3). The proof can be achieved by mimicking the three-interval method given for the Morse--Bott case in the contact manifold situation (see \cite{oh-wang}) under the normal form given in Section \ref{sec:darboux}.  We leave details to readers. 
In particular, the optimal decay rate for $\delta$ is in fact the inverse of the order of the holonomy $\Hol_{\eta_\infty}$, which is $1$ if assuming the action is free on $\mu^{-1}(0)$. 
\item The technical point which forces us to obtain the higher order exponential decay as in the current proof instead of directly applying the Morse--Bott condition from the result in Proposition \ref{prop:hess}, is that one is not able to obtain a convergent subsequence to localize the operator or the functional near a critical point from the a priori estimates (\cite[Lemma 9.3]{gaio-salamon}) for the case $b\not\equiv 0$. 

This trouble disappears for the standard cylindrical metric, i.e., the case $b\equiv 0$, and an alternative proof direct to the Morse--Bott condition was given in \cite{L2} by the authors. 
\end{enumerate}
\end{rem}

\bigskip

{\bf Acknowledgments:}
Bohui Chen and Bai-Ling Wang would like to thank RIMS and in particular Professor Kaoru Ono for the hospitality and excellent research environment. Rui Wang would like to thank Professor Yong-Geun Oh for his continuous support and collaborations over the years. 

%
%
%
%

\end{document}